\documentclass[11pt]{article}

\usepackage{etex}
\usepackage[margin={1.2in}]{geometry}
\usepackage[titletoc]{appendix}

\usepackage{tikz}
\usetikzlibrary{matrix}
\usepackage{eucal}

\usepackage{mathrsfs}
\usepackage{stmaryrd}
\usepackage{amscd}
\usepackage{amsthm}

\usepackage{rotating}

\usepackage{longtable}

\usepackage[sc]{mathpazo}
\usepackage{courier}

\usepackage[utf8]{inputenc}
\usepackage[T1]{fontenc}

\usepackage[english]{babel}
\usepackage{blindtext}

\usepackage{amsmath}

\usepackage{microtype}

\usepackage{amsmath}
\usepackage{amssymb}
\usepackage{amsthm}
\usepackage{color}
\usepackage{latexsym,extarrows}

\usepackage{graphicx}
\graphicspath{ {images/} }

\usepackage[all]{xy}

\usepackage[stable,multiple]{footmisc}

\RequirePackage[font=small,format=plain,labelfont=bf,textfont=it]{caption}
\makeatletter
\newsavebox{\@brx}
\newcommand{\llangle}[1][]{\savebox{\@brx}{\(\m@th{#1\langle}\)}%
  \mathopen{\copy\@brx\kern-0.5\wd\@brx\usebox{\@brx}}}
\newcommand{\rrangle}[1][]{\savebox{\@brx}{\(\m@th{#1\rangle}\)}%
  \mathclose{\copy\@brx\kern-0.5\wd\@brx\usebox{\@brx}}}
\makeatother

\begin{document}
\def\e#1\e{\begin{equation}#1\end{equation}}
\def\ea#1\ea{\begin{align}#1\end{align}}
\def\eq#1{{\rm(\ref{#1})}}
\theoremstyle{plain}
\newtheorem{thm}{Theorem}[section]
\newtheorem*{thm*}{Theorem}
\newtheorem{lem}[thm]{Lemma}
\newtheorem{prop}[thm]{Proposition}
\newtheorem{cor}[thm]{Corollary}
\newtheorem*{cor*}{Corollary}
\theoremstyle{definition}
\newtheorem{dfn}[thm]{Definition}
\newtheorem{ex}[thm]{Example}
\newtheorem{rem}[thm]{Remark}
\newtheorem{conjecture}[thm]{Conjecture}

\newcommand{\D}{\mathrm{d}}
\newcommand{\A}{\mathcal{A}}
\newcommand{\bL}{\mathbb{L}}
\newcommand{\LL}{\llangle[\Big]}
\newcommand{\RR}{\rrangle[\Big]}
\newcommand{\LD}{\Big\langle}
\newcommand{\bR}{{\mathbb{R}}}
\newcommand{\RD}{\Big\rangle}
\newcommand{\F}{\mathcal{F}}
\newcommand{\HH}{\mathcal{H}}
\newcommand{\X}{\mathcal{X}}
\newcommand{\PP}{\mathbb{P}}
\newcommand{\K}{\mathscr{K}}
\newcommand{\bK}{\mathbb{K}}
\newcommand{\q}{\mathbf{q}}

\newcommand{\op}{\operatorname}
\newcommand{\C}{\mathbb{C}}
\newcommand{\bC}{\mathbb{C}}
\newcommand{\cC}{\mathcal{C}}
\newcommand{\tcC}{{\widetilde{\cC}}}
\newcommand{\N}{\mathbb{N}}
\newcommand{\R}{\mathbb{R}}
\newcommand{\bT}{\mathbb{T}}
\newcommand{\Q}{\mathbb{Q}}
\newcommand{\Z}{\mathbb{Z}}
\newcommand{\bZ}{\mathbb{Z}}
\newcommand{\NE}{\mathrm{NE}}
\newcommand{\bP}{\mathbb{P}}
\newcommand{\cO}{\mathcal{O}}

\renewcommand{\H}{\mathbf{H}}

\newcommand{\si}{\sigma}
\newcommand{\Si}{\Sigma}

\newcommand{\Etau}{\text{E}_\tau}
\newcommand{\E}{{\mathcal E}}
\newcommand{\G}{\mathbf{G}}
\newcommand{\eps}{\epsilon}
\newcommand{\g}{\mathbf{g}}
\newcommand{\im}{\op{im}}

\newcommand{\h}{\mathbf{h}}

\newcommand{\Gmax}[1]{G_{#1}}
\newcommand{\AW}{E}
\providecommand{\abs}[1]{\left\lvert#1\right\rvert}
\providecommand{\norm}[1]{\left\lVert#1\right\rVert}
\newcommand{\abracket}[1]{\left\langle#1\right\rangle}
\newcommand{\bbracket}[1]{\left[#1\right]}
\newcommand{\fbracket}[1]{\left\{#1\right\}}
\newcommand{\bracket}[1]{\left(#1\right)}
\newcommand{\ket}[1]{|#1\rangle}
\newcommand{\bra}[1]{\langle#1|}

\newcommand{\ora}[1]{\overrightarrow#1}

\providecommand{\from}{\leftarrow}
\newcommand{\bl}{\textbf}
\newcommand{\mbf}{\mathbf}
\newcommand{\mbb}{\mathbb}
\newcommand{\mf}{\mathfrak}
\newcommand{\mc}{\mathcal}
\newcommand{\cinfty}{C^{\infty}}
\newcommand{\pa}{\partial}
\newcommand{\prm}{\prime}
\newcommand{\dbar}{\bar\pa}
\newcommand{\OO}{{\mathcal O}}
\newcommand{\hotimes}{\hat\otimes}
\newcommand{\BV}{Batalin-Vilkovisky }
\newcommand{\CE}{Chevalley-Eilenberg }
\newcommand{\suml}{\sum\limits}
\newcommand{\prodl}{\prod\limits}
\newcommand{\into}{\hookrightarrow}
\newcommand{\Ol}{\mathcal O_{loc}}
\newcommand{\mD}{{\mathcal D}}
\newcommand{\iso}{\cong}
\newcommand{\dpa}[1]{{\pa\over \pa #1}}
\newcommand{\Kahler}{K\"{a}hler }
\newcommand{\0}{\mathbf{0}}

\newcommand{\B}{\mathcal{B}}
\newcommand{\V}{\mathcal{V}}

\newcommand{\M}{\mathfrak{M}}
\newcommand{\fp}{\mathfrak{p}}
\newcommand{\fg}{\mathfrak{g}}

\renewcommand{\Im}{\op{Im}}
\renewcommand{\Re}{\op{Re}}

\newcommand{\DD}{\Omega^{\text{\Romannum{2}}}}


\newtheorem{assumption}[thm]{Assumption}
\newcommand{\rd}{{\mathrm{res}}}
\newcommand{\bF}{{\mathbb F}}
\newcommand{\tN}{\widetilde{N}}
\newcommand{\btau}{\boldsymbol{\tau}}
\newcommand{\tF}{{\widetilde{F}}}
\newcommand{\Eff}{{\mathrm{Eff}}}
\newcommand{\tX}{{\widetilde{X}}}
\newcommand{\one}{{\mathbf{1}}}
\newcommand{\tM}{\widetilde{M}}
\newcommand{\orb}{{\mathrm{CR}}}
\newcommand{\BoxS}{\mathrm{Box}(\mathbf \Si)}
\newcommand{\age}{{\mathrm{age}}}
\newcommand{\vir}{{\mathsf{vir}}}
\newcommand{\Def}{{\mathrm{Def}}}
\newcommand{\Aut}{{\mathrm{Aut}}}
\newcommand{\CR}{{\mathrm{CR}}}
\newcommand{\pt}{{\mathrm{pt}}}
\newcommand{\Obs}{{\mathrm{Obs}}}
\newcommand{\bQ}{{\mathbb{Q}}}
\newcommand{\cQ}{{\mathcal{Q}}}
\newcommand{\Nef}{\mathrm{Nef}}
\newcommand{\cI}{{\mathcal{I}}}
\newcommand{\tNef}{{\widetilde{\mathrm{Nef}}}}
\newcommand{\fr}{{\mathfrak{r}}}
\newcommand{\fs}{{\mathfrak{s}}}
\newcommand{\fm}{{\mathfrak{m}}}
\newcommand{\cL}{{\mathcal {L}}}
\newcommand{\ev}{\mathrm{ev}}
\newcommand{\w}{{\mathsf{w}}}
\newcommand{\cT}{{\mathcal{T}}}
\newcommand{\tNE}{{\widetilde{\mathrm{NE}}}}
\newcommand{\cX}{{\mathcal{X}}}
\newcommand{\bmu}{{\mathbold{\mu}}}
\newcommand{\bSi}{{\partial{\Si}}}
\newcommand{\vmu}{ {\vec{\mu}} }
\newcommand{\Mbar}{\overline{\mathcal{M}}}
\newcommand{\cU}{{\mathcal{U}}}
\newcommand{\cW}{{\mathcal{W}}}


\numberwithin{equation}{section}

\title{\bf Open Gromov-Witten Theory of $K_{\bP^2}, K_{{\bP^1}\times {\bP^1}}, K_{W\bP[1,1,2]}, K_{\bF_1}$ and Jacobi Forms}
\author{Bohan Fang, Yongbin Ruan, Yingchun Zhang and Jie Zhou}
\date{}
\maketitle

\begin{abstract}
It was known through the efforts of many works that the generating functions in the closed Gromov-Witten theory of $K_{\bP^2}$
 are meromorphic quasi-modular forms
\cite{Coates:2014fock, Lho:2017, Coates:2018} basing on the B-model predictions \cite{Bershadsky:1993cx, Aganagic:2006wq, Alim:2013eja}. In this article, we extend the modularity phenomenon to $K_{{\bP^1}\times {\bP^1}}, K_{W\bP[1,1,2]}, K_{\mathbb F_1}$. More importantly, we
generalize it to the generating functions in the open Gromov-Witten theory using the theory of Jacobi forms where the open Gromov-Witten parameters are transformed into
elliptic variables.
\end{abstract}

\setcounter{tocdepth}{2} \tableofcontents

\section{Introduction}

Toric Calabi-Yau (CY) manifolds/orbifolds have always occupied a special place in geometry and physics. The combinatorial nature of these objects make them a fertile ground to test new ideas and techniques. In principle, the Gromov-Witten (GW) theory of such a space has been computed in the 90's by the localization technique. However, the solution is in terms of a complicated graph sum formula which makes it not so useful for actual computations. Then, the topological vertex formalism was developed in early 2000 \cite{Aganagic:2003db}. In this past decade, a new formulation of its B-model in terms of mirror curve leads to the {\em Remodeling Conjecture} \cite{BKMP2009} via topological recursion \cite{Eynard:2007invariants}. This conjecture
for toric CY's has been proved by the first author and his collaborators \cite{FLZ16}. However, the topological recursion formalism computes the open GW invariants via a recursion algorithm. From the mathematical point of view, the ultimate goal is to compute its generating functions by closed formula in some sense. These generating functions are quite complicated and it is rare that they can be expressed as elementary functions.
The next attractive classes of functions are modular forms from number theory. Indeed, a great deal of efforts were spent to show that the generating functions of closed GW invariants of $K_{\bP^2}$ are meromorphic quasi-modular forms \cite{Coates:2014fock, Lho:2017, Coates:2018} basing on the earlier results in \cite{Bershadsky:1993ta, Bershadsky:1993cx, Klemm:1999gm, Chiang:1999tz, Yamaguchi:2004bt,  Aganagic:2006wq, Grimm:2007tm, Alim:2008kp, Haghighat:2008gw, Alim:2013eja}. \\

The main purpose of this article is to push further the interaction between GW theory and modular forms to the cases $K_{{\bP^1}\times {\bP^1}}, K_{W\bP[1,1,2]}, K_{\mathbb F_1}$. More importantly, we generalize it to open GW theory. Open GW theory has an additional open parameter keeping track of the number of boundaries. Our key idea is to replace meromorphic quasi-modular forms by certain "quasi-meromorphic Jacobi forms" (see Section \ref{sec:jacobi} for the precise definition), while the open parameter is translated into a certain function of the elliptic variable $z$ of the latter.

Let's briefly recall the definition of Jacobi forms here. A meromorphic function $\Phi$ on $\mathbb{C}\times \mathcal{H}$ is a meromorphic Jacobi form of weight $k\in \mathbb{Z}$,  index $\ell\in \mathbb{Z}$ for the modular group $\Gamma(1)=SL_{2}(\mathbb{Z})$ if it satisfies
\begin{itemize}
		\item
		$\Phi({z\over c\tau+d}, {a\tau+b\over c\tau+d})
		=(c\tau+d)^{k} e^{2\pi i  \ell  {cz^{2}\over c\tau+d}} \Phi(z,\tau)\,,
		\quad
		\forall
		\begin{pmatrix}
		a & b\\
		c& d
		\end{pmatrix}\in \Gamma(1)\
		$\,,
		\item
		$\Phi(z + m\tau+n, \tau)
		=e^{-2\pi i \ell (m^2 \tau+2m z)}
		\Phi(z,\tau)\,, 	\quad\forall m,n\in \mathbb{Z}
		$\,.
\end{itemize}
together with some regularity condition. It is ``modular'' in $\tau$ and ``elliptic'' in $z$.
See Section \ref{sec:jacobi} (e.g., Definition \ref{dfnJacobiforms}) for detailed definitions on holomorphic and meromorphic Jacobi forms for a modular subgroup $\Gamma<SL_{2}(\mathbb{Z})$. The set of all such meromorphic Jacobi forms gives a graded ring $\mathcal J(\Gamma)$. Adjoining the quasi-modular Eisenstein series $E_2$, we get the graded ring of \emph{quasi-meromorphic Jacobi forms}.

\begin{ex}
The Weierstrass $\wp$-function
\begin{equation}
\wp(u,\tau)={1\over u^2}+\sum_{(m,n)\in\mathbb{Z}\oplus\mathbb{Z}\textbackslash\{(0,0)\}}\Big({1\over (u+m\tau+n)^2}-{1\over (m\tau+n)^2}\Big)\,
\end{equation}
is a meromorphic Jacobi form of weight $2$, index $0$ for the modular group $SL_2(\mathbb{Z})$.
\end{ex}

The main result of this article is about the generating series $F_{g,n}$ of open GW invariants (called open-closed GW potential)--for its definition see \eqref{eqn:Fgn} below. It collects open-closed GW invariants of genus $g$ with $n$ boundary components into a generating series. When $n=0$ this is the usual closed GW potential and is denoted by $F_g$. The quantity $F_{g,n}$ is a formal power series of the closed parameters $Q=(Q_{1},\cdots)$ and open parameters $X=(X_1,\dots,X_n)$.  In our four examples, $Q=Q_1$ or $Q=(Q_1, Q_2)$, corresponding to the K\"ahler parameters $T_{1}$ or $(T_{1},T_{2})$ by $Q_{k}=e^{T_{k}}, k=1,2$.

Recall that the mirror map can be derived within GW theory by using Givental's $I$-functions. In our cases, it is  a bi-holomorphic map
$\mathbf{m}: \Delta_{Q}\rightarrow \Delta_{q}$
from a certain neighborhood $\Delta_{Q}$ of the large radius limit $Q=0$ in $\mathbb A^1$ or $\mathbb{A}^{2}$ to such an one $\Delta_{q}$ of the large complex structure limit $q=(q_{1},q_{2})=0$ in $\mathbb A^1$ or  $\mathbb{A}^{2}$.
The parameters $q$ appear naturally as complex parameters in the defining equations of the mirror curve family
$\chi: \mathcal{C}\rightarrow \mathcal{U}_{\mathcal{C}}$
of the toric CY, see \cite{Hori:2000kt}.
Induced by the mirror map $\mathbf{m}$, the parameter $X_k$ gets identified with a function of a rational function $x_k$ on the mirror curve. They are explicit hypergeometric-like series (see for example \eqref{eqnopenmodulus}) with nice leading order behavior \eqref{eqn:mirror-map}.

As will be discussed in Section \ref{sec:one-parameter}, when there are two K\"ahler parameters we need to restrict to a non-trivial one-parameter family
so that we can employ the theory of modular forms.
Namely, we choose a rational affine curve $\mathcal{U}_{\rd}$ in the base $\mathcal{U}_{\mathcal{C}}$ of the mirror curve family whose Zariski closure includes $q=0$, then
we take the fiber product to get the restriction of the family $\chi_{\rd}: \mathcal{C}_{\rd}\rightarrow \mathcal{U}_{\rd}$.
From the perspective of the A-model, this corresponds to the restriction to the preimage under $\mathbf{m}$
of the (analytification) of the subvariety
$\mathcal{U}_{\rd} \cap \Delta_{q}$.

After the restriction to the one-parameter family, $q_1$ and $q_2$ are modular functions (for a certain modular subgroup $\Gamma$
depending on $\chi_{\rd}$) in
the complex structure parameter $\tau$
for the mirror curve lying on the upper-half plane $\mathcal{H}$. 
In this way $Q$ also becomes a function of  $\tau$, although it is not modular (see e.g. Example \ref{exKP2continued}).
In fact $\tau$ has a purely A-model expression
$\tau=\partial_{t}^2F_0$,
where $t$ is a certain $\bZ$-linear combination of $T_1$ and $T_2$ (or simply $T_1$ for the one K\"ahler parameter case).
The parameter $t$ is called the flat coordinate for the one-parameter subfamily.
A different choice of such combination amounts to an $SL_{2}(\bZ)$-transformation on $\tau$ which still represents the same complex structure of the mirror curve.
See Section \ref{sec:one-parameter} for details on this.

We then make use of the uniformization to express the rational function $x$ on the mirror curve
as
$x(u,\tau)$, in terms of meromorphic modular and Jacobi forms.
Here $u\in \bC$ is the universal cover of the mirror curve, which is isomorphic to $\bC/(\bZ\oplus \tau \bZ)$ with  $\tau \in \mathcal H$. Thus one may regard
the formal series $F_{g,n}(Q,X_{1},\cdots X_{n})$ as one in $(\tau, u_1,\dots, u_n)$.

One of the main results of this article is to show that above $F_{g,n}(Q,X_{1},\cdots X_{n})$ for the examples $K_{\bP^2}, K_{{\bP^1}\times {\bP^1}}, K_{W\bP[1,1,2]}, K_{\bF_1}$ are quasi-meromorphic
Jacobi forms under the mirror map
$\mathbf{m}$.

\begin{thm*}[Theorem \ref{thmholhighergenusWgn}]
The following statements hold for $d_{X_1}\dots d_{X_n}F_{g,n}$ with $2g-2+n> 0$.

\begin{enumerate}
\item The differential $d_{X_1}\dots d_{X_n}F_{g,n}$ is a quasi-meromorphic Jacobi form of total weight $n$.
\item The closed GW potential $F_{g}$ is a meromorpic quasi-modular form of weight zero.
\end{enumerate}
\end{thm*}
Here we say the differential $d_{X_1}\dots d_{X_n}F_{g,n}$ is Jacobi if its coefficient
with respect to the basis $du_1\boxtimes\dots\boxtimes du_n$, which is a meromorphic function in any $u_{k}$, is Jacobi in $(u_k,\tau)$.

Part 2 of the above theorem applied to $K_{\bP^2}$ recovers the results in \cite{Lho:2017}
recently proved basing on the earlier results in e.g. \cite{Bershadsky:1993cx, Alim:2013eja}.
Restricting to the subfamily obtained by setting $T_{1}=T_{2}$ for the $K_{\mathbb{P}^{1}\times \mathbb{P}^{1}}$ case, it recovers the results in
\cite{Lho:2018} on this particular subfamily.\\

Applying some elementary properties of modular forms and Jacobi forms, we obtain the following corollary of the above theorem.
\begin{cor*}[Corollary \ref{corGWexpansionofWgn}]
	The Taylor coefficients of $F_{g,n} $ in a certain $X$-expansion
	are meromorphic quasi-modular forms.
	\end{cor*}

\begin{ex}[Proposition \ref{propdiskpotentialJacobi}]
	\label{diskpotentialex}
The disk potential $\partial_x{W}:=(\log y) /x$ involves the logarithm of a meromorphic Jacobi form. For the $K_{\mathbb{P}^2}$ case, this is
\begin{eqnarray*}
\partial_x{W}={\log\left(\kappa^{3}\wp '(u)+{3\over 2}(-4)^{1\over 3}\kappa^2\phi\wp(u)+{9\over 8}\phi^3-{1\over 2}\right)\over -3(-4)^{1\over 3}\kappa^{2} \phi \wp (u)-{9\over 4}\phi^{3}}\,.
\end{eqnarray*}
\end{ex}

\begin{ex}[Proposition \ref{lemannuluspotentialJacobi}]
	\label{annuluspotentialex}
The annulus potential is
\begin{eqnarray*}
\omega_{0,2}(u_1,u_2)=(\wp(u_1-u_2)+ {\pi^{2}\over 3} E_2)du_1\boxtimes du_2\,.
\end{eqnarray*}
It is a quasi-meromophic Jacobi form of weight $2$, index $0$.
\end{ex}

\begin{ex}[Theorem \ref{thmholhighergenusWgn}]
	\label{w03ex}
	For the $(g,n)=(0,3)$ case, one has
\begin{eqnarray*}
d_{X_1}d_{X_2}d_{X_3}F_{0,3}&=&\omega_{0,3}(u_{1},u_{2},u_{3})\nonumber\\
&=&\sum_{r\in R^{\circ}}  \left(  2[{1\over \Lambda} ]_{-2}\cdot
\prod_{k=1}^{3}
(\wp(u_k-u_r)+ {\pi^{2}\over 3} E_2)
\right) du_1\boxtimes du_2\boxtimes du_3\,.
\end{eqnarray*}
It is a quasi-meromorphic Jacobi form of total weight $3$, index $0$ for a certain modular group. For the $K_{\mathbb{P}^2}$ case, one has
\begin{eqnarray*}
[{1\over \Lambda} ]_{-2}=
{1\over \wp''^2(u_r)} {{1\over 2}\left(1-3
	(-4)^{1\over 3}\kappa^{2}   \phi \wp(u_r)-{9\over 4 }\phi^{3}\right)\over \kappa^{3}}  {(-4)^{1\over 3}\kappa^{2} \wp(u_r) +{3\over 4}\phi^{2} \over -2{ (-4)^{1\over 3}\kappa^{2}
}} \,.
\end{eqnarray*}
\end{ex}
In Example \ref{diskpotentialex} and Example \ref{w03ex} above,
\begin{equation*}
\phi(\tau)=\Theta_{A_{2}}(2\tau) {\eta(3\tau) \over \eta(\tau)^{3}  }\,,
\quad
\kappa=\zeta_{6}\,2^{-{4\over 3}} 3^{1\over 2} \pi^{-1}{\eta(3\tau)\over \eta(\tau)^{3}}\,,\nonumber
\end{equation*}
with $\Theta_{A_{2}}(2\tau)$, ${\eta(3\tau)  \eta(\tau)^{-3}}$ modular forms for the modular group $\Gamma_0(3)$, see \cite{Zagier:2008, Maier:2009, Maier:2011} for details.
Also
\begin{equation*}
R^{\circ}=\{u_{r}={1\over 2}, {\tau\over 2}, {1+\tau\over 2}\}\,,
\end{equation*}
and
\begin{eqnarray*}
\wp({1\over 2})={\pi^{2}\over 3} (\theta_{3}^4+\theta_{4}^4)\,,\quad
\wp({\tau\over 2})={\pi^{2}\over 3} (-\theta_{2}^4-\theta_{3}^4)\,,\quad
\wp({1+\tau\over 2})={\pi^{2}\over 3}  (\theta_{2}^4-\theta_{4}^4)\,,
\end{eqnarray*}
with $\wp''(u_{r})=6 \wp^{2}(u_{r})-{2\over 3}\pi^{4}E_{4}$.
More formulae for the $K_{\mathbb{P}^1 \times \mathbb{P}^1}$, $K_{W\mathbb{P}[1,1,2]}$ and $K_{\mathbb{F}_1}$ cases can be found in Appendix \ref{secappendix}.\\

With a structure theorem
in Theorem \ref{thmhighergenusWgn} in hand, we can combine the holomorphic anomaly equation of \cite{Eynard:2007holomorphic} with the above results and arrive at the Yamaguchi-Yau type equation. We refer to Theorem \ref{thmhae} for its full form, but just state the result for the closed sector here.

\begin{thm*}[Theorem \ref{thmhae}]
 The closed GW potentials $F_{g}$ $(g\geq 2)$ satisfy
\begin{equation*}
{\partial \over \partial \eta_{1}} F_{g}= {\partial_{Y} S^{00} \over \partial_{Y}\eta_{1}} \cdot {1\over 2}\left(\partial_{t}\partial_{t}F_{g-1}+\sum_{ g_1=1}^{g-1}\partial_{t}F_{g_{1}} \cdot \partial_{t}F_{g_{2} } \right)\,.
\end{equation*}
Here $\eta_1=(\pi^2/3)E_2$.
\end{thm*}
We refer to Section \ref{secholomorphicanomalyequations} for the definition of $S^{00}$ and $Y$, but only remark that in our cases $ {\partial_{Y} S^{00} / \partial_{Y}\eta_{1}} $ is a constant number. For example for $\cX=K_{\bP^2}$, this is $3/(2\pi^2$), see Example \ref{exKP2continued}. Our theorem recovers the Yamaguchi-Yau equation for $K_{\bP^2}$ as shown in \cite{Lho:2017}
recently proved basing on the earlier results in e.g. \cite{Bershadsky:1993cx, Alim:2013eja}.
Restricting to the subfamily obtained by setting $T_{1}=T_{2}$ for the $K_{\mathbb{P}^{1}\times \mathbb{P}^{1}}$ case, this theorem recovers the
Yamaguchi-Yau equation proved in \cite{Lho:2018} on this particular subfamily.

\subsection*{Outline of the proof}

We review toric geometry, mirror symmetry and the remodeling conjecture for our four examples in Section \ref{sec:review}. The four examples in the article are chosen for the fact that their mirror curves are genus one algebraic curves equipped
with hyperelliptic structures determined by the structure of the branes. The genus one and hyperelliptic structure allow us to apply some arithmetic geometry of elliptic curves in the study of topological recursion on the mirror curve. In Section \ref{sec:arithmetic} we show that the hyperelliptic structure implies the ramification points are identified with the group of $2$-torsion points on the elliptic curve. We also explicitly give uniformization results for the mirror curve families in Section \ref{sec:arithmetic}. In Section \ref{sec:proof}, an examination following the procedure in topological recursion then shows that the differentials $\{\omega_{g,n}\}_{g,n}$,
produced by residue calculus near the ramification points, are quasi-meromorphic Jacobi forms lying in a ring with very simple generators. A structure theorem of $\{\omega_{g,n}\}_{g,n}$, relating the weights, poles of the quasi-meromorphic Jacobi forms $\{\omega_{g,n}\}_{g,n}$ to the genus $g$ and number of boundary components $n$, follows by induction. This then offers a rigorous proof of the Yamaguchi-Yau type holomorphic anomaly equations for $d_{X_1}\dots d_{X_n}F_{g,n}$, basing on the equations \cite{Eynard:2007holomorphic} satisfied by $\{\omega_{g,n}\}_{g,n}$ and the proof in \cite{FLZ16}
stating that
$d_{X_1}\dots d_{X_n}F_{g,n}=\omega_{g,n}$.
Furthermore, on the mirror curve there is a distinguished point around which the expansions of GW potentials give rise to open GW invariants.
The results on uniformization imply that the Taylor coefficients at this point of the GW potentials which are now regarded as quasi-meromorphic Jacobi forms, are meromorphic quasi-modular forms.

We remark that the hyperelliptic structure is crucial in our discussion. There are 12 other local  toric CY 3-folds whose mirror curves are in hyperelliptic forms. In principle, our technique applies to these examples as well. At this point, we are unsure about the compatibility of their hyperelliptic structures and the remodeling conjecture-- we hope to come back in another time (see Remark \ref{rem:other-cases} for a more technical discussion). A more exciting future direction is the case of genus two mirror curve, in which the ramification data can be also made intrinsic from the hyperelliptic structure. We will leave it to another paper.

\subsection*{Acknowledgement}

B. F. would like to thank Chiu-Chu Melissa Liu and Zhengyu Zong for enlightening discussion.
J. Z. would like to thank Murad Alim, Florian Beck, Kathrin Bringmann, Xiaoheng Jerry Wang and Baosen Wu for useful discussions.
The authors are very grateful to the anonymous referee for the great improvement of this article.

Y. R. is partially supported by NSF grant DMS 1405245 and NSF FRG grant DMS 1159265.
Y. Z. is supported by China Scholarship Council grant No. 201706010026.
J. Z. 's work was done while he was a postdoc at the University of Cologne and
was partially supported by German Research Foundation Grant CRC/TRR 191.

\section{A brief review of the solution of Remodeling Conjecture}
\label{sec:review}

\subsection{Toric Calabi-Yau $3$-orbifolds and mirror curves}

Let $N'= \bZ^2$ and $N=N'\oplus \bZ$. Then $N'_\bR=N'\otimes_{\bZ} \bR\cong \bR^2$, and $N_\bR=N\otimes_{\bZ} \bR =N'_\bR \times \bR$. Consider the  triangulated polytopes in $N'_\bR$ in Figure \ref{fig:defining-polytope}. We embed $N'_\bR\hookrightarrow N_\bR$ as $N'_\bR \times\{1\}$. Then the cones over these triangulated polytopes with the vertex at $0\in N_\bR$ are automatically fans in the definition for toric varieties -- $3$ cones are cones over faces in the triangulation; $2$ cones are cones over edges while $1$-cones are cones over vertices. The origin $0$ in $N_\bR$ is the $0$-cone. The triangulations in Figure \ref{fig:defining-polytope} give rise to fans which define toric Calabi-Yau $3$-orbifolds $\cX=K_{\mathbb P^2}, K_{{\bP^1}\times {\bP^1}}, K_{W\bP[1,1,2]}, K_{\mathbb F_1}$.
\begin{figure}[h]
  \centering{
  \resizebox{150mm}{!}{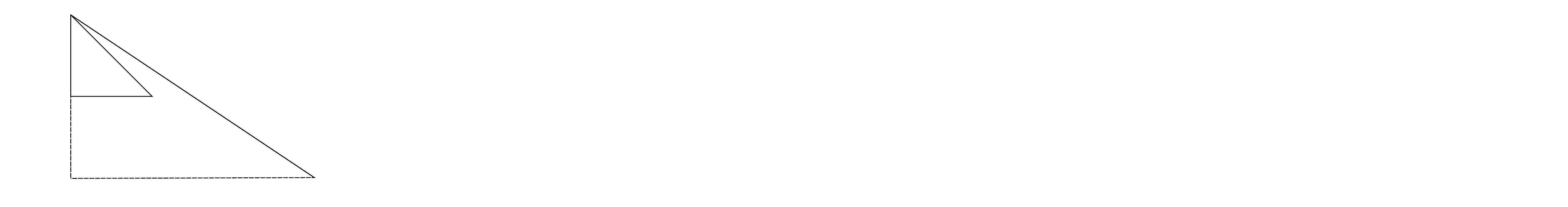}
  \caption{The defining polytopes of $K_{ \bP^2}$, $K_{\bP^1\times \bP^1}$, $K_{W\bP[1,1,2]}$ and $K_{\mathbb F_1}$ respectively. The polytopes are in gray and their triangulations are in solid lines. They are fitted into a triangle in dashed lines, which ensures that the mirror curves are in hyperelliptic forms.}
  \label{fig:defining-polytope}
  }
\end{figure}

We denote the defining polytope by $\Delta$. We notice that all of these polytopes are contained in a triangle with vertices $(0,-1),(0,1),(4,-1)$ -- the goal is to ensure that the mirror curves are in hyperelliptic form (see \eqref{eqnHxy=0} for the mirror curve equation). The toric orbifolds are given by the fan data as a cone at the origin in $\bR^3$ over these triangulated polytopes embedded inside $\bR^2\times\{1\} \subset \bR^3$. The orbifolds given in this way is automatically CY. See \cite{Cox:2011} for more detailed definition of toric varieties and orbifolds.

We let $\cX$ be one of these orbifolds, and $\bT\cong (\bC^*)^3$ be the dense algebraic torus inside $\cX$. The Calabi-Yau torus $\bT'\subset \bT$ preserves the Calabi-Yau forms. By the construction of the toric orbifold $\bT'= \N'\otimes_{\bZ} \bC^*$. Let $\pi'_\bR: \cX\to \bR^2$ be the moment map of is maximal compact subgroup $\bT'_\bR$. Let $\cX^1$ be the union of the $\bT'$-invariant $1$-suborbifolds of $\cX$, which are either weighted projective lines or gerbes over $\bC$. The toric graph of $\cX$ is the image $\pi'_\bR(\cX^1)$. It is a trivalent graph, and the images of gerby $\bC$ are rays while the images of weighted projective lines are segments. Each vertex is the image of a $\bT'$-fixed stacky point.

\begin{figure}[h]
  \centering{
  \resizebox{40mm}{!}{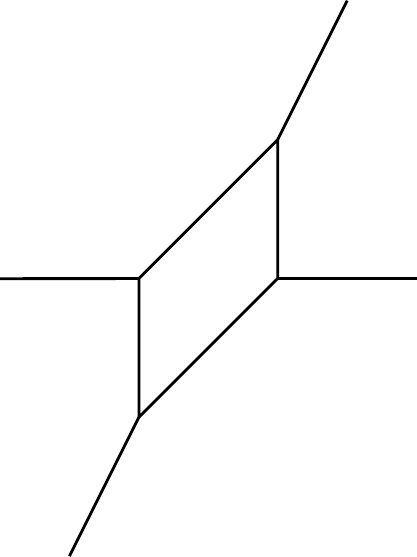}
  \caption{The toric graph of $K_{\bP^1\times \bP^1}$. The Aganagic-Vafa brane $\mathcal L$ is in the inverse image of $\pi_\bR^{\prime -1}(\mathrm{pt})$.  }
  \label{fig:toric-graph}
  }
\end{figure}

In this article, we will only consider so called {\em outer Aganagic-Vafa Lagrangian brane $\mathcal L$}. It is in the inverse image $\pi_\bR^{\prime-1}(\mathrm{pt})$, where the point $\mathrm{pt}$ is on an outer leg of the toric diagram, as shown in Figure \ref{fig:toric-graph}. With an extra condition that when presenting $\cX$ as a GIT quotient $\bC^{\fp+3}\sslash(\bC^*)^\fp$ the sum of $\fp+3$ complex coordinates on $\cL$ is a constant, $\cL$ is a Lagrangian submanifold in $\cX$ diffeomorphic to $\bR^2\times S^1$.

The \emph{framing} of an Aganagic-Vafa brane is simply a choice of $f\in \bZ$, which determines a one-dimensional subtorus $\bT'_f$ of the two-dimensional torus $\bT'$: let $\w_1=(1,0)$ $\w_2=(0,1)$ be lattice points in $M'=\mathrm{Hom}(N',\bZ)$, and thus characters of $\bT'$, then define $\bT'_f=\mathrm{ker}(\w_2-f\w_1)$. Together $(\cL,f)$ is a framed Aganagic-Vafa brane.

Given a toric CY 3-fold with a framed Aganagic-Vafa brane, there is a standard procedure to write down its mirror curves. Let $(m_1,n_1),\dots,(m_{\fp+3},n_{\fp+3})$ be integral points in the defining polytope. For our examples, all defining polytopes have a triangle inside with vertices $(1,0),(0,1)$ and $(0,0)$. By a permutation, we require $(m_i,n_i)=(1,0),(0,1),(0,0)$ for $i=1,2,3$ respectively. The affine mirror curve equation is then
\begin{equation}
\label{eqnHxy=0}
H(x,y,q)=x+y+1+\sum_{i=1}^\fp a_i(q) x^{m_{i+3}}y^{n_{i+3}}=0.
\end{equation}
This is an affine curve in $(\bC^*)^2$, denoted by $C^\circ$. It has a natural compactification into a compactified mirror curve $C$ in the toric orbifold $\bP_\Delta$. The polytope $\Delta$ defines a toric variey $\mathbb P_\Delta$ and an ample line bundle $L_\Delta$ on it. The standard construction of $\bP_\Delta$ is the Zariski closure of
\begin{equation}
\bP_\Delta=\overline{\{[x:y:1:x^{m_4}y^{m_4}:\dots:x^{m_{\fp+3}}y^{n_{\fp+3}}]\}}\subset \bP^{\fp+2}.
\label{eqn:PDelta}
\end{equation}
Here $H$ should  be regarded as a section of $L_\Delta$. Then the zero set $C$ is the compactification of $C^\circ$. We denote the Zariski open subset $\mathcal U_\cC\subseteq  (\bC^*)^\fp$ on which $C^\circ$ and $C$ are smooth. When $q\in \mathcal U_\cC$, the topological quantities of the mirror curve are recapped from the toric data, where $3+\fp$ is the number of integer points in $\Delta$, while $\fg$ is the number of interior integer points in $\Delta$:
\begin{equation}
h_1(C^\circ)=\fg+\fp+2,\ h_1(C)=2\fg\,.
\end{equation}
Meanwhile, $h^2_\CR(\cX)=\fp$ and $h^4_\CR(\cX)=\fg$, where $h^i_\CR(\cX)$ is the dimension of $H^i_\CR(\cX)$, the $i$-th Chen-Ruan cohomology of $\cX$. In particular, the genus of $C$ is $\fg$. Therefore we have a family of smooth compactified mirror curves $\cC$ over $\mathcal U_\cC$ with fiber $C$ and $\cC^\circ \subset \cC$ is the family of smooth affine mirror curves with fiber $C^\circ$.

We list our four examples in details below.
\begin{ex}
  \label{ex:1}
  $\cX=K_{ \bP^2}$ (the canonical bundle over $\bP^2$). The affine mirror curve $C^\circ$ is
  \begin{equation}
  x+y+1+q_{1} x^3/y=0.
  \end{equation}
  The compactified mirror curve $C$ sits in $ \bP_\Delta=\bP^2/\bmu_3$.
\end{ex}

\begin{ex}
  \label{ex:2}
  $\cX=K_{\bP^1\times\bP^1}$ (the canonical bundle over $\bP^1\times \bP^1$).
  The affine mirror curve $C^\circ$ is
  \begin{equation}
  x+y+1+q_1 x^2 + q_2 x^2/y=0.
  \end{equation}
  The compactified mirror curve $C$ sits in $ \bP_\Delta=(\bP^1\times \bP^1 )/\bmu_2$.
\end{ex}

\begin{ex}
  \label{ex:3}
  $\cX=K_{W\bP[1,1,2]}$ (the canonical bundle over $W\bP[1,1,2]$).
  The affine mirror curve $C^\circ$ is
  \begin{equation}
  x+y+1+q_1 x^2+q_2 x^4/y=0.
 \end{equation}
  The compactified mirror curve $C$ sits in $\bP_\Delta=W\bP[1,1,2]/\bmu_2$.
\end{ex}

\begin{ex}
  \label{ex:4}
  $\cX=K_{\mathbb F_1}$.  The affine mirror curve $C^\circ$ is
  \begin{equation}
1+x+y+q_1x y^{-1}+q_2x^2=0\,.
\end{equation}
\end{ex}

For a generic choice of parameters $q=(q_{1},q_{2},\cdots,q_\fp) \in \cU_\cC$, the affine mirror curve $C^\circ$ of $\cX$ is holomorphic Morse with respect to the covering $x:C^\circ \to \bC^*$. For our examples $\cX=K_S$ for $S=\bP^2, \bP^1\times \bP^1, W\bP[1,1,2], \mathbb F_1$, the number of ramification points is $3$ for $S=\bP^2$ and $4$ for others -- which is the same as the $\fp+\fg+1=\dim H^*_\CR(\cX)$.
We denote
by $R$ the divisor
 of ramification points of
$x:C\to \mathbb{P}^{1}$
on the mirror curve $C$ (those on $C^{\circ}$ are called finite ramification points).

\subsection{Mirror symmetry from remodeling conjecture}
\label{sec:remodeling}

Let's briefly review the definition of open GW theory Stable maps to orbifolds with Lagrangian boundary conditions and their moduli spaces have been introduced in \cite[Section 2]{CP}. Let $(\Si,x_1,\ldots, x_n)$ be a prestable bordered orbifold Riemann surface with $n$ interior marked points in the sense of \cite[Section 2]{CP}. Then the coarse moduli space $(\bar{\Si},\bar{x}_1,\ldots, \bar{x}_n)$ is a prestable bordered Riemann surface
with $n$ interior marked points, defined in \cite[Section 3.6]{Katz:2001} and \cite[Section 3.2]{liu02}.
We define the topological type $(g,h)$ of $\Si$
to be the topological type of $\bar{\Si}$ (see \cite[Section 3.2]{liu02}).

Let $(\Si,\bSi)$ be a prestable bordered orbifold Riemann surface of type $(g,h)$,
and let $\bSi = R_1\cup \cdots \cup R_h$ be union of connected component. Each connected
component is a circle which contains no orbifold points. Let $\varphi:(\Si,\bSi)\to (\cX,\cL)$ be a (bordered) stable map in the sense of \cite[Section 2]{CP}. The topological type of $\varphi$ is given by the degree $\beta'= { \varphi}_*[\Si]\in H_2(\cX,\cL;\bZ)$ and $ \mu_i=\varphi_*[R_i]\in H_1(\cL;\bZ)$. Given $\beta'\in H_2(\cX,\cL;\bZ)$ and
\begin{equation}
\vmu= (\mu_1,\dots, \mu_h) \in H_1(\cL;\bZ)^h.
\end{equation}
Let $\Mbar_{(g,h),n}(\cX,\cL\mid \beta',\vmu)$ be the moduli space of stable
maps of type $(g,h)$, degree $\beta'$, winding numbers and twisting $\vmu$,
with $n$ interior marked points.

There are evaluation maps (at interior marked points)
\begin{equation}
\ev_j:\Mbar_{(g,h),n}(\cX,\cL\mid \beta',\vmu) \to \cI\cX,\quad j=1,\ldots,n,
\end{equation}
where $\cI\cX$ is the inertia stack of $\cX$.

Let $\bT'_\bR\cong U(1)^2$ be the maximal compact subgroup of $\bT'\cong (\bC^*)^2$. Then
the $\bT'_\bR$-action on $\cX$ is holomorphic and preserves $\cL$, so it acts on
the moduli spaces $\Mbar_{(g,h),n}(\cX,\cL \mid\beta',\vmu)$.
Let $H^{\bT'_\bR}_{\CR}(\cX;\bQ)$ be the $\bT'$-equivariant Chern-Ruan cohomology of $\cX$. Given $\gamma_1,\ldots, \gamma_n \in H^*_{\bT',\orb}(\cX;\bQ)=H^*_{\bT'_\bR,\orb}(\cX;\bQ)$, we define
\begin{equation}
\langle \gamma_1,\ldots, \gamma_n\rangle^{\cX,\cL,\bT_\bR'}_{g,\beta',\vmu}:=
\int_{[F]^{\vir}} \frac{\prod_{j=1}^n (\ev_j^*\gamma_i)|_F}{e_{\bT_\bR'}(N^\vir_F)} \in \bQ(\w_1,\w_2)
\end{equation}
where $F\subset \Mbar_{(g,h), n}(\cX,\cL\mid\beta',\vmu)$ is the $\bT_\bR'$-fixed points set of
the $\bT_\bR'$-action on $\Mbar_{(g,h),n}(\cX,\cL\mid\beta',\vmu)$ and $\bQ(\w_1,\w_2)$ is the fractional field of $H^*_{\bT_\bR'}(\pt;\bQ)\cong \bQ[\w_1,\w_2]$. As shown in \cite{fang-liu-tseng}, when all $\gamma_i \in H^2(\cX)$, it turns out that the open GW invariant $\langle \gamma_1,\ldots, \gamma_n\rangle^{\cX,\cL,\bT_\bR'}_{g,\beta',\vmu}\in \bQ(\frac{\w_2}{\w_1})$, and specifying a framing $f$ amounts to setting $\frac{\w_2}{\w_1}=f$.  Then for $\gamma_1,\dots, \gamma_n \in H^2(\cX)$ the following open GW invariant
\begin{equation}
\langle \gamma_1,\ldots, \gamma_n\rangle^{\cX,(\cL,f) }_{g,\beta',\vmu}:=\iota^*_{\bT'_f\to \bT'}(\langle \gamma_1,\ldots, \gamma_n\rangle^{\cX,\cL,\bT_\bR'}_{g,\beta',\vmu})\in \bQ
\end{equation}
where $\iota^*_{\bT'_f\to \bT'}: H^*_{\bT'}(\cX)\to H^*_{\bT'_f}(\cX)$ is the induced map on the equivariant cohomology. For the rest of this paper we only consider the framing zero $f=0$ case, namely setting the equivariant parameters $\w_2=0$ and $\w_1=1$, and simply write $\langle \dots \rangle^{\cX,\cL}$ for $\langle \dots\rangle^{\cX,(\cL,0)}$.

Let $w\in H_1(\cL;\bZ)$. It also represents a class in $H_2(\cX,\cL;\bZ)$ given by the holomorphic disk $u\mapsto u^w, |u|<r$ inside the $\bT$-invariant $1$-dimensional ($\cong \bC$) represented by the outer leg (c.f. Figure \ref{fig:toric-graph}). This gives a splitting of the exact sequence
\[
0\to H_2(\cX;\bZ)\to H_2(\cX,\cL;\bZ)\to H_1(\cL;\bZ) \cong \bZ\to 0.
\]
Then for $d\in H_2(\cX;\bZ)$ and $w\in H_1(\cL;\bZ)$ we can write $d+w\in H_2(\cX,\cL;\bZ)$. Define the open-closed GW potential
\begin{eqnarray}
  \label{eqn:Fgn}
&&F_{g,n}^{\cX,\cL}({\bold  T}; X_1,\ldots, X_n) \nonumber \\
&=&\sum_{d\in \Eff(\cX)}\sum_{\mu_1,\ldots,\mu_n>0}\sum_{\ell\geq 0}
\frac{\langle \mathbf T^\ell \rangle^{\cX,\cL}_{g,d+\sum{\mu_i},(\mu_1,k_1),\ldots, (\mu_n,k_n)}}{\ell !} \cdot
\prod_{j=1}^n (X_j)^{\mu_j}\,,
\end{eqnarray}
where $\mathbf T= T_1 H_1+\dots T_\fp H_\fp$, where $\{H_a\}$ is the integral basis of $H^2_\CR(\cX)$ and they are in the extended K\"ahler cone of $\cX$.\footnote{For a complete treatment of extended K\"aher cone in the toric orbifold setting, we refer to \cite[Section 3.1]{iritani09}. It coincides with the actual K\"ahler cone when $\cX$ is a smooth manifold.   In general the extended K\"ahler cone is not necessarily a simplicial cone, but in our four examples they are.} We further require $H_1,\dots, H_{\fp'}\in H^*(\cX)\subset H^*_\CR(\cX)$ and $H_{\fp'}, \dots, H_{\fp}$ to be in the \emph{twisted sector}. We have $\fp'=\fp$ when $\cX$ is smooth. When $n=0$ this becomes the closed GW potential and we denote it by $F_{g}^\cX(\mathbf T)$. For simplicity we use $F_{g,n}$ and $F_g$ for  $F^{\cX,\cL}_{g,n}$ and $F_g^\cX$ respectively.

The remodeling conjecture of \cite{BKMP2009,BKMP2010} relates open-closed GW potential $F^{\cX,\cL}_{g,n}$ of $(\cX,\cL)$ to Eynard-Orantin's topological recursion invariants $\omega_{g,n}$. The full version of this conjecture, including the orbifold cases, is proved in \cite{FLZ16}.

Let $K_1(C^\circ;\bC)=\mathrm{ker}(H_1(C^\circ;\bC)\to H_1((\bC^*)^2,\bC))\cong \bC^{\fg+\fp}$ where the map $H_1(C^\circ;\bC)\to H_1((\bC^*)^2,\bC)$ is induced from
 \begin{equation}
 C^\circ\stackrel{(x,y)}{\longrightarrow}(\bC^*)^2.
 \end{equation}

The enumerative mirror symmetry is corrected by the mirror map. We refer the reader to \cite[Section 4.1 and 4.2]{fang-liu-tseng} for the explicit form of the mirror map. We only list its asymptotic behavior here:
\begin{align}
    \label{eqn:mirror-map}
  \nonumber
  T_a&=\log q_a + o(q), a=1,\dots,\fp'\,,\\
  T_a&=q_a+ o(q), a=\fp'+1,\dots,\fp\,,\\
  X_i&=x_i(1+{O}(q))\, .\nonumber
\end{align}
Notice that since we have fixed $q_a$ as in Examples \ref{ex:1}, \ref{ex:2}, \ref{ex:3} and \ref{ex:4}, this particular asymptotic behavior determines each $H_a$.

By the explicit construction in \cite[Section 5.5]{FLZ16}, the closed mirror maps are given by period integrals. Over a neighborhood of $q=0$ in $\cU_\cC$ (although $0$ is taken away from $\cU_\cC$ since the mirror curve is not smooth there), there are (local families of) cycles $\tilde A_1,\dots, \tilde A_\fp\in K_1(C^\circ;\bC)$  such that
\begin{equation}
T_a=\int_{\tilde A_a}\log y\frac{dx}{x}\,,\quad  a=1,\dots,\fp.
\end{equation}
They are called \emph{closed mirror maps}. The integrals are well-defined up to constants since the cycles are in $K_1$.

We define the bifundamental, a.k.a. Bergmann kernel, $\omega_{0,2}$ as follows.
\begin{itemize}
  \item $\omega_{0,2}$ is a symmetric meromorphic form on $C^2$, with the only pole at the diagonal, i.e. for any local coordinate $z$
  \begin{equation}
  \omega_{0,2}(z_{1},z_{2})=\frac{dz_1 \boxtimes dz_2}{(z_1-z_2)^2}+\text{holomorphic part}.
  \end{equation}
  \item We require
  \begin{equation}
  \int_{z\in \bar A_a}\omega_{0,2}(z,w)=0\,, \quad  a=1,\dots,\fp\,.
  \end{equation}
\end{itemize}
That is, $\omega_{0,2}$ vanishes on the $A$-cycles of a Torelli's marking. Here each $\bar A_a$ is the image of $\tilde A_a$ when passing to $H_1(C;\bC)$. Notice that $\{\bar A_a\}_{a=1}^\fp$ span a Lagrangian subspace in $H_1(C;\bC)$.

The fundamental bidifferential $\omega_{0,2}$ constructed in the previous subsection depends on the choice of a Lagrangian subspace of $H_1(C;\bC)$ on curves over a neighborhood of $0$ in $\cU_\cC$, \emph{which cannot be extended over the entire $\cU_\cC$}. Let $A_1,\dots,A_\fg,B_1,\dots,B_\fg\in H_1(C;\bZ)$ be a Torelli's marking, i.e. $A_i\cap A_j=B_i\cap B_j=0$ and $A_i\cap B_j=\delta_{ij}$. Define the coordinates $\tau_{ij}$.
\begin{equation}
  \label{rem:schiffer}
\int_{A_j}\omega_i=\delta_{ij}\,,\quad \int_{B_j}\omega_i=\tau_{ij},
\end{equation}
where $\{\omega_i\}$ form a basis in $\Omega^1(C)$.
One can then define the modified cycles following \cite{Eynard:2007invariants}
\begin{equation}
  \label{eqn:modified-cycles}
\underline A_i=A_i-\sum_j {\kappa_{ij}} (B_j-\sum_{l=1}^\fg \tau_{jl}A_l)\,,\quad \underline B_i=B_i-\sum_{j=1}^\fg A_j,
\end{equation}
where $\kappa=-1/{(\tau-\bar\tau)}$ which is a $\fg\times \fg$ matrix. Notice that in all our four examples $\fg=1$. They also form a Lagrangian subspace in $H_1(C;\bC)$. We use the cycles $\underline A_1,\dots,\underline A_\fg$ to define the bidifferential $\widehat{ \omega}_{0,2}$, this is
what is called the \emph{Schiffer kernel}, see \cite{Tyurin:1978periods}.
It turns out that $\widehat{\omega}_{0,2}$ does not depend on the choice of $A_1,\dots,A_\fg$, and is \emph{defined for curves over the entire $\cU_\cC$.}
From the definition, by regarding $\mathrm{Im}\tau_{ij}$ as formal variables, then the Schiffer kernel has the "holomorphic limit"
\begin{equation}
\lim_{\mathrm{Im}\tau_{ij}\to \infty}\widehat{ \omega}_{0,2}=\omega_{0,2}\,.
\end{equation}
See \cite{Fay:1977fourier, Tyurin:1978periods, Takhtajan:2001free} for details.

Define $R^{\circ}$ to be the ramification locus of $x:C^\circ\to \mathbb{P}^{1}$. We assume $R^{\circ}$ has only simple ramifications which is the case of our examples for generic $q\in \cU_\cC$ . The Eynard-Orantin topological recursion defines $\omega_{g,n}$ $(2g-2+n>0)$ recursively as follows
\begin{align}
  \label{eqn:EO-recursion}
  \omega_{g,n}(p_1,\ldots, p_n) &= \sum_{p_0\in R^{\circ}}\mathrm{Res}_{p \to p_0}
  \frac{\int_{\xi = p}^{\bar{p}} B(p_n,\xi)}{2(\lambda(p)-\lambda({p^*}))}
  \Big( \omega_{g-1,n+1}(p,{p^*},p_1,\ldots, p_{n-1}) \\
  \nonumber
  &\quad\quad  + \sum_{g_1+g_2=g}
  \sum_{ \substack{ J\cup K=\{1,..., n-1\} \\ J\cap K =\emptyset } } \omega_{g_1,|J|+1} (p,p_J)\omega_{g_2,|K|+1}({p^*},p_K)\Big),
\end{align}
where $\lambda =\log y\frac{d  x}{ x}$, $\omega_{0,1}=0$, and for any $p\in C^\circ$ around a simple ramification point $p_0$, $p^* \neq p$ is the unique point that has the same $x$-coordinate.
Eynard-Orantin has shown that $\omega_{g,n}$ is symmetric, and at most has poles at ramification points.\\

The mirror  $C$ has a distinguished point
\begin{equation}\label{eqnopenGWpoint}
\mathfrak s_0=(x,y)=(0,-1)\,.
\end{equation}
We call this the open large radius limit point or open GW point.

\begin{thm}[Open-sector remodeling conjecture and disk mirror theorem \cite{fang-liu-tseng, FLZ16} restricted to our cases, in framing zero]
  \label{thm:remodeling-open}
  Under the open-closed mirror map \eqref{eqn:mirror-map}, for $\cX=K_S$ where $S=\bP^2, \bP^1\times \bP^1, W\bP[1,1,2], \bF_1$,  we have the following statements.
  \begin{itemize}
    \item When $2g-2+n>0$,
      \begin{equation}
      d_{X_1}\dots d_{X_n} F_{g,n}= (-1)^{g-1+n}\omega_{g,n}\,.
      \end{equation}
    \item When $(g,n)=(0,2)$,
    \begin{equation}
    d_{X_1}d_{X_2}F_{0,2}=\omega_{0,2}-\frac{d
   x_1  \boxtimes d x_2}{( x_1-x_2)^2}\,.
    \end{equation}
    \item When $(g,n)=(0,1)$,
    \begin{equation}
    ( x\frac{\partial}{\partial x})^2 F_{0,1}={ x} \frac{\partial }{\partial x}\log y\,.
    \end{equation}
    We understand $\omega_{g,n}$ or $\log y$ as expansions by power series in $ x_1,\dots,  x_n$ (or $ x$) at the open large radius limit point
    $\mathfrak s_0=(0,-1)$. Notice that $\omega_{0,2}$ is singular at the diagonal so we need to subtract its singular part first.
  \end{itemize}
\end{thm}

\begin{thm}[Closed-sector remodeling conjecture for $g>1$, \cite{FLZ16}, restricted to our cases]
  \label{thm:remodeling-closed}
 Under the same assumption as the previous theorem,
  \begin{equation}
  \label{eqn:Fg}
  F_g=\frac{1}{2g-2}\sum_{p_0\in R^{\circ}}\mathrm{Res}_{p\to p_0}\left( d^{-1} \lambda(p)   \cdot  \omega_{g,1}(p)\right)\,,
  \end{equation}
  where $d^{-1}\lambda=\int\lambda$, which can be locally defined near each ramification point, and the constant ambiguity does not affect the result.
\end{thm}
\begin{rem}
\label{rem:other-cases}
The remodeling conjecture is proved for any toric CY $3$-orbifold under \emph{generic framing} in \cite{FLZ16}. The essential requirement is that the number of ramification points of $xy^{-f}$ for an integer framing $f$
is the same as $\dim H^*_\CR(\cX)$. In previous sections, we state remodeling conjecture in \emph{framing zero}, where the number of ramification points of $x$ is the same as $\dim H^*(\cX)$ for $K_S$ being our four cases. In other cases, the affine mirror curves, if in hyperelliptic forms, have $3$ or $4$ $x$-ramification points and which is \emph{less} than the dimension of $H^*_\CR(\cX)$.
\end{rem}

Once we replace the Bergman kernel $\omega_{0,2}$ by the Schiffer kernel $\widehat \omega_{0,2}=S$, we denote the recursion result by $\widehat\omega_{g,n}$. Furthermore we use $\widehat F_g$ to denote the right hand side of \eqref{eqn:Fg} after replacing $\omega_{0,2}$ by $\widehat \omega_{0,2}$.  We have $\lim_{\mathrm{Im}\tau\to\infty} \widehat\omega_{g,n}=\omega_{g,n}$ and $\lim_{\mathrm{Im}\tau\to\infty} \widehat F_g=F_g$.

\section{Geometry of genus one mirror curves}
\label{sec:arithmetic}

Hereafter by a curve $C$ we mean a smooth projective variety over $\mathbb{C}$ of pure dimension one. Our technique requires the affine mirror curve
\eqref{eqnHxy=0} to be in hyperelliptic form. That is, the equation of the curve \eqref{eqnHxy=0} can be transformed by a bi-regular morphism into the form
\begin{equation}\label{eqnhyperellipticformofmirrorcurve}
\tilde{y}^2=g(x)
\end{equation}
after the simple change of variables
\begin{equation}\label{eqnsimplechangeonybyh}
\tilde{y}=y+h(x)
\end{equation}
where $h(x)$ is a quadratic polynomial in $x$. In particular one writes
\begin{equation}\label{eqnshyperellipticinvolutionalgebraic}
y^*=-y-2h(x)\,,
\end{equation}
then the action $*:(x,y) \mapsto (x, y^*)$ gives the hyperelliptic involution.

The remodeling conjecture relates $\omega_{g,n}$ to $F^{\cX,\cL}_{g,n}$. The main tool of our investigation is the hyperelliptic form of the mirror curve $\tilde{y}^2=g(x)$, for which the ramification points of $x$ have very nice properties.

\subsection{Basic definitions on modular forms and Jacobi forms}
\label{sec:jacobi}

In topological recursion,
we need to represent various ingredients in terms of
modular forms and Jacobi forms.\\

We now give very quick definitions, without explaining many of the subtitles. See e.g., \cite{Zagier:2008} for a quick introduction to modular forms.
Readers who are familiar with these concepts can skip this subsection.

\begin{dfn}[Modular forms]\label{dfnmodularforms}
	A holomorphic function $\phi:  \mathcal{H}\rightarrow \mathbb{C}$, where $\mathcal{H}$ is the upper-half plane, is a called a (holomorphic) modular form of weight $k\in \mathbb{Z}$
	for the modular
	group $\Gamma<SL_{2}(\mathbb{Z})$ if it satisfies
	\begin{enumerate}
		\item
		$\phi({a\tau+b\over c\tau+d})
		=(c\tau+d)^{k} \phi(\tau)\,,
		\quad
		\forall \gamma=
		\begin{pmatrix}
		a & b\\
		c& d
		\end{pmatrix}\in \Gamma\
		$\,,
		\item
		$\phi$ has sub-exponential
		growth at infinity $\tau\rightarrow i \infty$.
	\end{enumerate}
\end{dfn}

Throughout this work we always assume that $\Gamma$ contains $-\mathbb{1}$. We shall encounter the Hecke subgroups $\Gamma_{0}(n)$, the subgroups $\Gamma_{1}(n)$, and the principal congruence subgroups $\Gamma(n)$ and their intersections, homogenized by adjoining $-\mathbb{1}$ if needed.
See the standard textbooks \cite{Rankin:1977ab, Schoeneberg:2012elliptic} for the precise definitions of these groups.

The factor $(c\tau+d)^{k}$ is called the automorphy factor.
In this paper, we shall need to work with modular forms with non-trivial multiplier systems.
The definition is as follows.
Fix a modular group $\Gamma<SL_{2}(\mathbb{Z})$.
A function $v: \Gamma\rightarrow \mathrm{U}(1)$ is called a multiplier system of weight $k$ for $\Gamma$ if it satisfies $v(-\mathbb{1})=(-1)^{k}$ and
\begin{equation}
v(\gamma_{1}\gamma_{2}) = w(\gamma_{1},\gamma_{2})v(\gamma_{1})v(\gamma_{2})\,,
\quad
\forall \gamma_{1},\gamma_{2}\in \Gamma
\end{equation}
for some function $w$ valued in $\{\pm  1\}$ making $v(\gamma)(c\tau+d)^{k}$ into
an automorphy factor. Replacing the automorphy factor $(c\tau+d)^{k}$ in Definition
\ref{dfnmodularforms} by $v (\gamma) (c\tau+d)^{k}$
one defines modular forms of weight $k$ with respect to the multiplier system $v$.

We say $v$ is a trivial multiplier system for $\Gamma$
if there exists a subgroup $\bar{\Gamma}<SL_{2}(\mathbb{Z})$
such that $\Gamma$ is generated by $\bar{\Gamma}$ and $-\mathbb{1}$
and that $v(\gamma)=1$ for any $\gamma\in \bar{\Gamma}$.
The simplest nontrivial multiplier system
is $v(\gamma)=\chi(d)$, where $\chi$ the extension to $\mathbb{Z}$
of some Dirichlet character $\chi: (\mathbb{Z}/N\mathbb{Z})^{*}\rightarrow \mathbb{C}^{*}$
satisfying $\chi(-1)=(-1)^{k}$. We then call the corresponding nontrivial multiplier system
$v$ a Dirichlet multiplier sytem, otherwise we call it non-Dirichlet.
See \cite{Rankin:1977ab, Schoeneberg:2012elliptic} for further details.\\

One can also define the variants quasi-modular forms and almost-holomorphic modular forms \cite{Kaneko:1995}.
\begin{dfn}[Quasi-modular forms]\label{dfnquasimodularforms}
	A holomorphic function $\phi:  \mathcal{H}\rightarrow \mathbb{C}$, where $\mathcal{H}$ is the upper-half plane, is called a quasi-modular form of weight $k\in \mathbb{Z}$
	for the modular
	group $\Gamma<SL_{2}(\mathbb{Z})$ if it satisfies
	\begin{enumerate}
		\item
		There exist holomorphic functions  $f_{j}:  \mathcal{H}\rightarrow \mathbb{C},j=1,2,\cdots k$,
		such that
		\begin{equation*}
		\phi({a\tau+b\over c\tau+d})
		=(c\tau+d)^{k} \phi(\tau)+\sum_{j=1}^{k} c^{j}(c\tau+d)^{k-j}f_{j}\,,
		\quad
		\forall
		\begin{pmatrix}
		a & b\\
		c& d
		\end{pmatrix}\in \Gamma\,,
		\end{equation*}
		\item
		$\phi$ has subexponential
		growth at infinity $\tau\rightarrow i \infty$.
	\end{enumerate}
\end{dfn}
\begin{dfn}[Almost-holomorphic modular forms]\label{dfnalmostholomorphicmodularforms}
	A real analytic function $\phi:  \mathcal{H}\rightarrow \mathbb{C}$, where $\mathcal{H}$ is the upper-half plane, is a called an  almost-holomorphic modular form of weight $k\in \mathbb{Z}$
	for the modular
	group $\Gamma<SL_{2}(\mathbb{Z})$ if it satisfies
	\begin{enumerate}
		\item
		$\phi({a\tau+b\over c\tau+d},\overline{({a\tau+b\over c\tau+d})} )
		=(c\tau+d)^{k} \phi(\tau,\bar{\tau})\,,
		\quad
		\forall
		\begin{pmatrix}
		a & b\\
		c& d
		\end{pmatrix}\in \Gamma\
		$\,,
		\item
		$\phi$ has polynomial growth in $1/ \mathrm{Im} \tau$
		as $\tau\rightarrow i \infty$.
	\end{enumerate}
\end{dfn}
A typical example of a quasi-modular form is the Eisenstein series $E_{2}$ of weight $2$
and of an almost-holomorphic modular form is
\begin{equation}
\widehat{E_{2}}=E_{2}+{-3\over \pi} {1\over \mathrm{Im}\tau}\,.
\end{equation}
In fact, we have the following structure theorem due to \cite{Kaneko:1995}.
The set of
modular forms, quasi-modular forms, almost-holomorphic modular forms
for the modular group $\Gamma$ form graded rings.
Denote them by $M(\Gamma), \widetilde{M}(\Gamma), \widehat{M}(\Gamma)$, respectively.
Then
\begin{equation}\label{eqnstructuretheoremofquasialmost}
\widetilde{M}(\Gamma)=M(\Gamma)[E_{2}]\,,
\quad
\widehat{M}(\Gamma)=M(\Gamma)[\widehat{E_{2}}]\,.
\end{equation}
For the full modular group $\Gamma(1)=SL_{2}(\mathbb{Z})$ one has $M(\Gamma(1))=\mathbb{C}[E_{4},E_{6}]$, where $E_{4},E_{6}$
are the Eisenstein series of weight $4,6$ respectively.\\

The above definitions generalize to the corresponding objects with multiplier systems.
For the examples later studied in this paper, all of the modular forms with non-trivial multiplier systems
arise from uniformization of some elliptic curve families (see Section \ref{secallfourexamples}) and
are usually explicit functions\footnote{In this paper we follow the convention in \cite{Zagier:2008} for the $\theta$-constants.} of $\theta$-constants or $\eta$-functions whose multiplier systems
are very explicit.
By passing to a smaller modular subgroup if necessary, we can assume that
their multiplier systems are trivial.
For this reason, we will usually be able to ignore the subtlety on multiplier systems in this work.
This will be demonstrated in all of our examples in Section \ref{secallfourexamples}.\\

We shall also encounter the notion of Jacobi forms \cite{Eichler:1984} whose definition is given as follows.
\begin{dfn}[Jacobi forms]\label{dfnJacobiforms}
	A holomorphic function $\Phi:  \mathbb{C}\times \mathcal{H}\rightarrow \mathbb{C}$ is a (holomorphic) Jacobi form of weight $k\in \mathbb{Z}$
	, index $\ell\in \mathbb{Z}_{>0}$ for the modular
	group $\Gamma<SL_{2}(\mathbb{Z})$ if it is "modular in $\tau$ and elliptic
	in $z$” in the sense that
	\begin{enumerate}
		\item
		$\Phi({z\over c\tau+d}, {a\tau+b\over c\tau+d})
		=(c\tau+d)^{k} e^{2\pi i  \ell  {cz^{2}\over c\tau+d}} \Phi(z,\tau)\,,
		\quad
		\forall
		\begin{pmatrix}
		a & b\\
		c& d
		\end{pmatrix}\in \Gamma\
		$\,,
		\item
		$\Phi(z + m\tau+n, \tau)
		=e^{-2\pi i \ell (m^2 \tau+2m z)}
		\Phi(z,\tau)\,, 	\quad\forall m,n\in \mathbb{Z}
		$\,.
		\item together with some regularity condition: in the Fourier expansion
		\begin{equation*}
		\Phi(z,\tau)=\sum_{n,r} c(n,r)e^{2\pi i n \tau }e^{2\pi i r z }\,,
		\end{equation*}
		one has
		$c(n,r)=0$ unless $4\ell n \geq r^2$.
	\end{enumerate}
\end{dfn}
See \cite{Eichler:1984} and the more recent work \cite{Dabholkar:2012} for details on variants of Jacobi forms
and the structure theorems of the rings they form.

For the purpose of this work, essentially we shall only use the weak Jacobi forms for the full modular group $SL_{2}(\mathbb{Z})$.
Weak Jacobi forms are defined by replacing the regularity condition in Definition \ref{dfnJacobiforms} by: $c(n,r)=0$ unless $n\geq 0$.
For $\Gamma(1)=SL_{2}(\mathbb{Z})$, the
set of all weak Jacobi forms $J$, bigraded by $(k,\ell)$, forms a $M(SL_{2}(\mathbb{Z}))$-module
\begin{equation}\label{eqnringofweakJacobiforms}
J=\mathbb{C}[E_{4},E_{6}][A,B,C]/ \langle   432C^2-AB^3+3E_{4}A^3 B-2E_{6}A^{4} \rangle \,,
\end{equation}
where $E_{4},E_{6}$ are the usual Eisenstein series, and $A,B,C$
are weak Jacobi forms of weight and index
$(k,\ell)=(-2,1), (0,1), (-1,2)$ respectively.
See for example \cite{Dabholkar:2012} for details. \\

In this work,
we shall need to work with "meromorphic modular forms", "meromorphic Jacobi forms"
which are defined with the requirement of holomorphicity replaced by meromorphicity.
To be more precise, fix a modular group $\Gamma$, we
denote the ring of meromorphic modular forms by $\mathcal{M}(\Gamma)$.
This is the fractional field (respecting the grading) of ring $M(\Gamma)$ of
(holomorphic) modular forms.
We also denote the factional field  (respecting the bigrading) of the ring $J$ of weak Jacobi forms for the full modular group $SL_{2}(\mathbb{Z})$ by $\mathcal{J}$.
It includes in particular the Weierstrass elliptic functions $\wp$ and $\wp'$
which are proportional to $B/A, C/A^2$ respectively according to e.g.,  \cite{Dabholkar:2012}.\\

We now introduce the following definitions, borrowing the terminologies "quasi" and "almost" from the
$\tau$-part of the corresponding functions.

\begin{dfn}[Quasi-meromorphic Jacobi forms and almost-meromorphic Jacobi forms]
	\label{dfnquasialmostmeromorphicJacobiforms}
	Fix a modular group $\Gamma<SL_{2}(\mathbb{Z})$.
	We define the ring of meromorphic quasi-modular forms, and almost-meromorphic modular forms to be
	\begin{equation}\label{eqnfractionalfieldadjoinedE2}
	\widetilde{\mathcal{M}}(\Gamma)=\mathcal{M}(\Gamma)[E_{2}]\,,
	\quad
	\widehat{\mathcal{M}}(\Gamma)=\mathcal{M}(\Gamma)[\widehat{E_{2}}]\,,
	\end{equation}
	where $\mathcal{M}(\Gamma)$ is the fractional field of the ring $M(\Gamma)$ of (holomorphic) modular forms for $\Gamma$.

	We define the ring of
	quasi-weak Jacobi forms and almost-weak Jacobi forms for the full modular group $SL_{2}(\mathbb{Z})$ to be
	\begin{equation}
	\widetilde{J}=J[E_{2}]\,,
	\quad
	\widehat{J}=J[\widehat{E_{2}}]\,,
	\end{equation}
	where $J$ is the rings of weak Jacobi forms for $SL_{2}(\mathbb{Z})$ in \eqref{eqnringofweakJacobiforms}.
	We define the ring of
	quasi-meromorphic Jacobi forms and almost-meromorphic Jacobi forms for the full modular group $SL_{2}(\mathbb{Z})$ to be
		\begin{equation}
	\widetilde{\mathcal{J}}=\mathcal{J}[E_{2}]\,,
	\quad
	\widehat{\mathcal{J}}=\mathcal{J}[\widehat{E_{2}}]\,,
	\end{equation}
	where $\mathcal{J}$ is the fractional field of the ring $J$ of weak Jacobi forms for $SL_{2}(\mathbb{Z})$.\\

In the rest of the paper we shall only work with meromorphic, quasi-meromorphic and almost-meromorphic Jacobi forms as opposed to the corresponding "weak" counterparts.\\

	For the modular group $\Gamma<SL_{2}(\mathbb{Z})$, we define\footnote{These definitions could be made more general by replacing $\mathcal{J}$ by the ring of meromorphic Jacobi forms for the modular subgroup $\Gamma$, but the former are already good enough for the purpose of this work. } the (multi-)graded rings of
	quasi-meromorphic Jacobi forms and almost-meromorphic Jacobi forms for the modular group $\Gamma<SL_{2}(\mathbb{Z})$ to be
	\begin{equation}\label{eqndfnquasiandalmostholomorphicJacobiforms}
	\widetilde{\mathcal{J}}(\Gamma)=\mathcal{J} \otimes \widetilde{\mathcal{M}}(\Gamma)
	=\widetilde{\mathcal{J}} \otimes \mathcal{M}(\Gamma)\,,
	\quad
	\widehat{\mathcal{J}}(\Gamma)=\mathcal{J}  \otimes \widehat{\mathcal{M}}(\Gamma) =\widehat{\mathcal{J} } \otimes \mathcal{M}(\Gamma)\,.
	\end{equation}
\end{dfn}

Following \cite{Kaneko:1995}, one can define the so-called "constant term map".
It is defined by regarding $\mathrm{Im}\tau$ as a formal variable and then sending it to infinity, hence the same as the "holomorphic limit".
It has the effect of
replacing $\widehat{E_2}$ by $E_{2}$ and of mapping the "almost" objects to "quasi" objects.
We also denote this operation by  $\lim_{\mathrm{Im}\tau\rightarrow \infty}$.\\

The final definition that we need is the notion of multi-variable meromorphic Jacobi form and
its variants.
Similar to the single-variable case we have the following definition.
\begin{dfn}\label{dfnmultiJacobiform}
A meromorphic function $\Phi: (\mathbb{C}  \times  \mathcal{H} )^{n}\rightarrow \mathbb{C}$
is a multi-variable meromorphic Jacobi form
of weight $k=(k_1,k_2,\cdots, k_n)\in \mathbb{Z}^n$, index $\ell=(\ell_1,\ell_2,\cdots, \ell_n)\in
\mathbb{Z}_{>0}^n$
for the modular
group $\prod_{j=1}^{n}\Gamma_{j}<(\mathrm{SL}_2(\Z))^n$,
if
item 1 in Definition \ref{dfnJacobiforms} is replaced by
\begin{eqnarray*}
&&\Phi\left({z_1\over c_1\tau_1+d_1}, {z_2\over c_2\tau_2+d_2}, \cdots, {z_n\over c_n\tau_n+d_n}, {a_1\tau_1+b_1\over c_1\tau_1+d_1},
{a_2\tau_2+b_2\over c_2\tau_2+d_2},
\cdots, {a_n\tau_n+b_n\over c_n\tau_n+d_n}
\right)\\
		&=&
		\prod_{j=1}^{n}(c_j\tau_j+d_j)^{k_j} e^{2\pi i  \ell_j  {c_j z_j^{2}\over c_j\tau_j+d_j}} \Phi(z_1,z_2,\cdots, z_n,\tau_1,\tau_2,\cdots, \tau_n)\,,
		\forall
		\begin{pmatrix}
		a_j & b_j\\
		c_j & d_j
		\end{pmatrix}\in \Gamma_j,j=1,2,\cdots n\,,
\end{eqnarray*}
and item 2 and 3 in Definition \ref{dfnJacobiforms} are replaced in the obvious way.
The total weight of $\Phi$ is defined to be the integer $\sum_{j=1}^{n}k_{j}$.
\end{dfn}

In this work, we restrict to the subclass of multi-variable meromorphic Jacobi forms consisting of polynomials of single-variable  meromorphic Jacobi forms in $(z_i,\tau_i)$, with $\tau_{1}=\tau_{2}=\cdots =\tau_{n}$ and
the modular group being $\prod_{j=1}^{n}\Gamma_{j}=\Gamma^{n}$ for some $\Gamma<\mathrm{SL}_2(\mathbb{Z})$.
This subclass again forms a ring which we still denote by $J$ by abuse of notation.
For this subclass, the notions of multi-variable quasi-meromorphic Jacobi forms and multi-variable almost-meromorphic Jacobi forms that we shall need are then defined in a way similar to the single-variable case.
Namely, we first define
$\mathcal{J}$ to be the fractional field of $J$,
then analogus to
\eqref{eqndfnquasiandalmostholomorphicJacobiforms}
we put
\begin{equation}
	\widetilde{\mathcal{J}}(\Gamma)=\mathcal{J} \otimes \widetilde{\mathcal{M}}(\Gamma)
	\,,
	\quad
	\widehat{\mathcal{J}}(\Gamma)=\mathcal{J}  \otimes \widehat{\mathcal{M}}(\Gamma) \,,
	\end{equation}
  where $\widetilde{\mathcal{M}}(\Gamma),   \widehat{\mathcal{M}}(\Gamma)
$
  are as introduced in \eqref{eqnfractionalfieldadjoinedE2}.

\subsection{Uniformizations of genus one algebraic curves}
\label{secuniformizationofellipticcurves}

In topological recursion, we shall need to explicitly express
many quantities, including rational functions on a curve $C$ of genus one, in terms
of modular and Jacobi forms.
\\

First let $C$ be a compact Riemann surface of genus one.
It is a classical fact that $C$ can be uniformized by the complex plane $\mathbb{C}$.
That is to say, there exits a lattice $\Lambda_{\tau}=\mathbb{Z}\oplus \mathbb{Z}\tau, \tau\in \mathcal{H}$ depending on $C$
such that $C$ is biholomorphic to  $\mathbb{C}/\Lambda_{\tau}$. For a fixed complex structure on $C$, there are $\mathrm{SL}_2(\bZ)$-many choices of $\tau$ related by the $\mathrm{SL}_2(\bZ)$-action $\tau \mapsto (a\tau+b)/(c\tau+d)$ for $(a,b;c,d)\in \mathrm{SL}_2(\bZ)$.
\begin{dfn}
A uniformization of the genus one compact Riemann surface $C$ is a biholomorphism of $C$ with $\mathbb C /\Lambda_{\tau}$ for some $\tau\in \mathcal H$. Such a uniformization is given by a universal cover $\pi: \bC\to C$, such that $\pi$ is holomorphic and that the preimage of any point $p$ in $C$ is $u(p)+\Lambda_\tau$ for some $u(p)\in \bC$ with $\tau$ depending on the complex structure of $C$.
\end{dfn}
In practice, the Abel-Jacobi map provides a uniformizing parameter $u$ on $C$. 
We may look the particularly interesting case where $C$ is in the Weierstrass normal form, which in the affine patch $Z=1$ of the plane $\mathbb{P}^{2}$ with homogeneous coordinates $[X,Y,Z]$ is given by $Y^2 =4X^3-a X -b$
for some $(a,b)\in \mathbb{A}^2$ such that the curve $C$ is smooth.
A uniformization of $C$ is provided by the Weierstrass elliptic functions
\begin{equation}\label{eqweierstrassuniformization}
X=\wp(u,\tau)\,,
\quad
Y=\partial_{u}\wp(u,\tau)\,.
\end{equation}
Here to obtain $\tau, u$ from $C$, we first choose a Torelli marking $\{A,B\}$ on the curve $C$ and a reference point $O$
for the Abel-Jacobi map $u$. We take
\begin{equation}
\label{eqn:tau-weierstrass}
\tau={\int_{B} {dX\over Y} \over \int_{A} {dX\over Y}}\,,
\quad
u(p)=\int_{O}^{p} {dX\over Y}\,,\quad
\forall p\in C\,.
\end{equation}
A universal covering map $\pi:\bC\to C$ can be given by  \eqref{eqweierstrassuniformization}.
A uniformization is then determined up to translation (inducing shift of origin in the group law on $C$) and scaling (inducing homothety on $C$) on $\mathbb{C}$.
The translation ambiguity can be fixed by requiring that the
origin $O$ in the group law to be $[0,1,0]$ for example, while the homothety can be uniquely determined by requiring
$a,b$ to be the following modular forms
\begin{equation}\label{eqng2g3asmodularforms}
g_{2}={4\over 3}{\pi^{4}E_{4}}\,,\quad
g_{3}={8\over 27}{\pi^{6}E_{6}}\,.
\end{equation}

We would also consider the family\footnote{This is a universal family with
	the base having a moduli stack interpretation. See e.g. \cite{Katz:1976p, Dubrovin:1994} for a nice account on this.} of Weierstrass normal forms
\begin{equation}
\mathcal{W}:\,Y^2 Z=4X^3-a X Z^2-b Z^3
\label{eqnWeierstrassnormalform}
\end{equation}
as a family of curves in $\mathbb P^2$ and defined over $\mathcal{U}_{\mathcal{W}}:=(\mathbb{A}^{2}-\{ a^{3}-27 b^2=0\})/\mathbb{C}^{*}$,
where the $\mathbb{C}^{*}$ acts on
$\mathbb P^2$ with weights $(2,3,1)$ and acts  on $\mathbb{A}^{2}$ with weights $(4,6)$. This family serves as the reference family for the construction of uniformization for families of curves of genus one.\\

Any smooth projective curve $C$ of genus one is isomorphic to a plane curve in Weierstrass normal form. A uniformization for $C$ can then be obtained by transforming the curve into the Weierstrass normal form, and then applying the aforementioned results for the latter, as we shall see through the examples in Section \ref{secallfourexamples}. For the cases that we are interested in, the family $\mathcal{C}$ is usually defined by a complete intersection in a weighted projective space of dimension
$N$
with small $N$. Practically, reducing the defining equations to the Weierstrass normal form can be done following the algorithms in e.g., \cite{Connell1996:elliptic}.

\subsection{Ramification points for hyperelliptic curves of genus one}

We will need to identify the ramification points for
a hyperelliptic cover $p: C\rightarrow \mathbb{P}^{1}$ of
a genus one curve $C$ as the 2-torsion points on its Jacobian. The statement is as follows.

\begin{lem}\label{lemramificationofhyperellipticgenusonecurves}
	Suppose the genus one curve $C$ is equipped with a hyperelliptic structure $p: C\rightarrow \mathbb{P}^{1}$.
	\begin{enumerate}
	\item
The set of ramification points $R$ are identified with the group of $2$-torsion points of the group law, with the origin of the group law chosen to be any of the ramification points.
	\item
	Under the Abel-Jacobi map with the reference point chosen to be any of the ramification points,
the involution on $C$ exchanging the two sheets of the hyperelliptic cover $p: C\rightarrow \mathbb{P}^{1}$ is induced by the map $u\mapsto -u$ on the Jacobian variety of $C$.
	\end{enumerate}
\end{lem}
\begin{proof}
	\begin{enumerate}
		\item
		Taking any two of the branch points $b_{1},b_{2}$,
		denote the corresponding ramification points by $r_{1},r_{2}$.
		Then
		we have for the divisor class
		\begin{equation}
		p^{*} ([b_{1}]-[b_{2}])=p^{*}([b_{1}])-p^{*}([b_{2}])=2[r_{1}]-2[r_{2}]
		=2([r_{1}]-[r_{2}])\,.
		\end{equation}
		Since the
		left hand side is principal, so is the right hand side $2([r_{1}]-[r_{2}])$.
		Then $[r_{1}]-[r_{2}]$ is a $2$-torsion on the Jacobian of $C$.

		Picking once and for all any of the ramification points makes the genus one curve $C$ an elliptic curve
		whose origin $O$ in the group law is the chosen point.
		By the property of the Abel-Jacobi map (with reference $O$) as an isomorphism, we see that the corresponding
		difference of $r_{1}, r_{2}$
		in the
		group law of the elliptic curve $C$ is a $2$-torsion point in the group law.

		\item
		Recall that the uniformization of the algebraic curve
		\eqref{eqnWeierstrassnormalform} and the Abel-Jacobi map $u$
		are related through the Weierstrass elliptic functions
		in
		\eqref{eqweierstrassuniformization}, with which the origin $O$ of the group law of the elliptic curve $C$ is mapped
		to $[0,1,0]$ in the homogenized coordinates of $[\wp,\wp',1]$.
		It is a classical fact that rational function field $k(C)$ of a genus 1 curve $C$ is generated by $\wp,\wp'$ with the algebraic relation given by the Weierstrass equation
		\begin{equation}\label{eqnalgebraicrelationwpwp'}
		k(C)\cong \mathbb{C}(\wp,\wp')/ \langle (\wp')^2-(4\wp^3-g_{2}\wp-g_{3})\rangle \,.
		\end{equation}
		The Galois group for the Galois extension $k(C)$ of the field $\mathbb{C}(\wp)$ is generated by
		$*: \wp\mapsto \wp, \wp'\mapsto -\wp'$.
		It is induced by
		the reflection $u\mapsto -u$ in the $u$-plane which is the universal cover of the elliptic curve $C$.

		We claim that the local involution around any ramification point of any hyperelliptic cover $p: C\rightarrow \mathbb{P}^{1}$
		of the genus one curve $C$ must be
		the above one.
		To see this, we simply observe that by analytic continuation this local involution determines an index $2$ rational subfield over $\mathbb{C}$.
		The fixed locus of this involution includes at least the ramification point.
		Up to isomorphism there is only one such index $2$ subfield, namely, $\mathbb{C}(\wp)$.
		This shows that the desired statement is true.

	\end{enumerate}

\end{proof}

\subsection{One-parameter subfamilies of genus one mirror curve families}
\label{sec:one-parameter}

In later discussions in topological recursion, we only consider the cases when $\cC$ is one of mirror curve families in Examples \ref{ex:1}, \ref{ex:2}, \ref{ex:3} and \ref{ex:4}.
Each of these families of genus one smooth projective curves is given by (the projective closure)
of the equation $H(x,y,q)=0$ in the toric Fano surface $\mathbb{P}_{\Delta}$ as shown in
 \eqref{eqnHxy=0}. We also take the hyperelliptic structure $p: C\rightarrow \mathbb{P}^{1}$
on any fiber $C$ in the family $\cC$ to be the hyperelliptic structure $x$ determined by the brane structure so that Lemma
\ref{lemramificationofhyperellipticgenusonecurves} applies.

In our four examples, we would like to restrict to one-parameter subfamilies of the families of mirror curves. Their bases are Zariski open subsets of $\mathbb{P}^{1}$. These one-parameter families studied in this work are obtained by specializing a possibly multi-parameter mirror curve family
$\chi: \mathcal{C}\rightarrow \mathcal{U}_{\mathcal{C}}$ to non-trivial one-parameter sub-families.
For $\cX=K_{\bP^2}$, $\cC$ is an one-parameter family, for which the base $\cU_\cC$ is actually the thrice punctured $\bP^1$. For the other cases $K_S, S=\bP^1\times \bP^1, W\bP[1,1,2], \bF_1$, the base $\cU_\cC$ is two-dimensional. We take a rational affine curve $\cU_\rd$ in $\cU_\cC$, such that the restriction of the family $\cC$ to $\cU_\rd$, denoted by $\cC_\rd$, has non-constant complex structures. Moreover, in the partial compactification of $\cU_\cC$ where the point $(q_1,\dots,q_\fp)=0$ is included, we require $0$ is also in the closure of $\cU_\rd$. Then we denote the one-parameter compactified mirror curve family by
$\chi_{\rd}:\cC_\rd\to \cU_\rd$, and the affine mirror curve family by $\chi_{\rd}^{\circ}:\cC_\rd^\circ \to \cU_\rd$.

We would like the following statement to be true.


\begin{assumption}
\label{lemwpuniformizationofoneparameterfamiliesgenusonecurves}
Suppose a non-trivial one-parameter family $\mathcal{C}_\rd\to \mathcal U_\rd$ of curves of genus one is
obtained as a subfamily of the mirror curve family in \eqref{eqnHxy=0}, and
is given by the equation $H(x,y,s)=0$. Then any rational function in $k(\mathcal{C}_\rd/\mathcal{U}_\rd)$ is a rational function of $\wp,\wp'$, with coefficients lying in the fractional field $\mathcal{M}(\Gamma)$ of the ring $M(\Gamma)$ of modular forms whose modular group $\Gamma$ depends on $\mathcal{C}_\rd$. Furthermore, there exits a holomorphic map $\mathbb C\times \mathcal H \to \cC_\rd$ given by $(u,\tau)\mapsto ((x(u,\tau),y(u,\tau),1),s(\tau))$, which fiberwisely gives a uniformization for $\mathcal{C}_{\rd,s(\tau)}\subseteq\mathbb P_\Delta$.
\end{assumption}

In the next section (Section \ref{secallfourexamples}), we will express (the generators of) rational functions over all these examples (and selected subfamilies when $\fp=2$) as modular forms in $\tau$, $\wp$ and $\wp'$ -- so this assumption is indeed true for our purposes.
\\

Let $\tilde A,\tilde B$ be cycles  in $K_1(C^\circ;\bZ)$ on a fiber  $C$ such that when passing to $H_1(C;\bZ)$, their images $\bar A, \bar B$ constitute a Torelli marking. We can recover the complex structure parameter $\tau$ of $C$ from
\begin{equation}
\frac{1}{2\pi\sqrt{-1}}\int_{\bar A} \lambda=t\,,\quad \frac{1}{2\pi\sqrt{-1}} \int_{\bar B}\lambda=t^B\,,\quad  \frac{\partial t^B}{\partial t}=\tau\,.
\end{equation}
This definition is compatible with \eqref{eqn:tau-weierstrass} and \eqref{rem:schiffer}. The parameter $t$ is called the \emph{flat} coordinate. The coordinate $t$ is equal to the K\"ahler parameter $T_1$ for $K_{\bP^2}$, or a linear combination of $T_1, T_2$ for the other cases. After restricting to $\cU_\rd$, all of $t$, $T_1$, $T_2$ are functions of $\tau$ with Assumption \ref{lemwpuniformizationofoneparameterfamiliesgenusonecurves}, which is true for all our examples. 

\subsection{Examples}
\label{secallfourexamples}

In this section, we give the uniformizations for the mirror curve families of
$K_{ \bP^2}$, $K_{\bP^1\times \bP^1}$, $K_{W\bP[1,1,2]}$ and $K_{\mathbb F_1}$, displayed in
Example \ref{ex:1}, \ref{ex:2}, \ref{ex:3}, \ref{ex:4}
respectively.
For each of these examples, the $\wp$-uniformization is derived by transforming the curve family $\mathcal{C}$ to the Weierstrass normal form
\eqref{eqnWeierstrassnormalform},
with the coordinates carefully so that the coefficients in the degree $1$ and $0$ terms in the resulting Weierstrass normal form
become exactly $-g_{2}, -g_{3}$ respectively.
The derivations are straightforward.

In all of our examples, the curve in the chosen affine patch is defined by the equation $(y+h(x))^{2}=g(x)$ as shown in
\eqref{eqnhyperellipticformofmirrorcurve} and \eqref{eqnsimplechangeonybyh}.
For the $K_{\mathbb{P}^{2}}$ and $K_{\mathbb{F}_{1}}$ cases, the degree of $g(x)$ is $3$.
Taking the origin $O$
for the group law to be the ramification point $\infty=[0,1,0]$ fixes the ambiguity in the shift $\epsilon$
of the argument in $\wp(u+\epsilon),\wp'(u+\epsilon)$ for the uniformization to be zero.
For the other cases, we choose once and for all a ramification point $O$ to be the origin.
Then in the rational functions $x(\wp,\wp'),y(\wp,\wp')$ in terms of $\wp(u+\epsilon),\wp'(u+\epsilon)$,
we have that $[x_{O}, y_{O},1]:=[x,y,1]|_{u=0}$ is the coordinate for the chosen ramification point $O$.
With these choices, the hyperelliptic involution is induced by $u\mapsto -u$ as shown in Lemma
 \ref{lemramificationofhyperellipticgenusonecurves}.

We shall also discuss the subtlety on multiplier systems mentioned in Section \ref{sec:jacobi}. The case by case analysis below, besides confirming Assumption \ref{lemwpuniformizationofoneparameterfamiliesgenusonecurves}, also proves the following lemma.
\begin{lem}
\label{lemmodularityofvaluesatramificationpoints}
	Consider the local toric Calabi-Yau 3-folds $\cX=K_{S}, S=\mathbb{P}^{2}, W\mathbb{P}[1,1,2], \mathbb{P}^{1}\times
	\mathbb{P}^{1}, \mathbb{F}_{1}$. Consider non-trivial one-parameter subfamilies of the mirror curves with  hyperelliptic structure determined by the corresponding brane.
	Then the generator of the rational functional field of the base  is a modular function in $\mathcal{M}( \Gamma)$, while
	the values of the rational functions $x,y$ at the ramification points are meromorphic modular forms in
	$\mathcal{M}(\Gamma(2)\cap \Gamma)$,
	for some modular group $\Gamma$ depending on the one-parameter family.
	\end{lem}
Note that here the modular group $\Gamma$ is not necessarily the maximal subgroup for which
the generator of the rational functional field of the base is modular.
Also we single out the role of the modular group $\Gamma(2)$ intentionally-- its appearance is due to the fact that
the ramification points are the 2-torsion points-- as we shall see below.

\subsubsection{$K_{\mathbb{P}^{2}}$}
\label{secexKP2}

The affine part of the mirror curve given in Example \ref{ex:1} is equivalent to
\begin{equation}
x^{3}+y^{2} +y -3\phi xy =0\,.
\end{equation}
The parameter $\phi$
is related to the parameter $q_{1}$ in  Example \ref{ex:1} by $q_{1}=(-3\phi)^{-3}$.
It is uniformized by
\begin{equation}
x=(-4)^{1\over 3}\kappa^{2} \wp (u)+{3\over 4}\phi^{2}\,,
\quad
y=\kappa^{3}\wp '(u)-({1-3
	\phi x\over 2})\,,
\end{equation}
with
\begin{equation}
\phi(\tau)=\Theta_{A_{2}}(2\tau) {\eta(3\tau) \over \eta(\tau)^{3}  }\,,
\quad
\kappa=\zeta_{6}\,2^{-{4\over 3}} 3^{1\over 2} \pi^{-1}\eta(3\tau)\eta(\tau)^{-3}\,,
\end{equation}
where $\Theta_{A_{2}}$ is the $\theta$-function for the $A_{2}$-lattice and $\eta$ is the $\eta$-function as a modular form.
The quantities $\phi, \kappa$
are modular forms for $\Gamma_{0}(3)$ with non-Dirichlet multiplier systems.
By passing to the smaller modular subgroup $\Gamma_{0}(9)$, we see that both
$\Theta_{A_{2}}(2\tau)$ and $\kappa$, and hence $\phi$,
are modular forms with the same quadratic multiplier system, which is given by the Dirichlet character
$\chi_{-3}$ taking the values $1,-1$ on $1,-1$ modulo $3$ respectively and zero otherwise.
See \cite{Borwein:1991, Borwein:1994, Berndt:1995, Maier:2009, Maier:2011} for details.
This confirms the discussion on multiplier systems following Definition \ref{dfnquasimodularforms} in Section \ref{sec:jacobi}.
By further passing to the subgroup $\Gamma= \Gamma_{0}(9)\cap\Gamma_{1}(3)$, all of them have trivial multiplier systems.

Under the
uniformization, the point $\infty=[0,1,0]$
corresponds to the origin $O$ of the group law, which is given by $u=0$ on the Jacobian.
The values of $x,y$ at the ramification points $u=1/2, {\tau/2}, {1+\tau/2}$
are meromorphic modular forms for
$\Gamma(2)\cap (\Gamma_{0}(9)\cap \Gamma_{1}(3))$,
by the standard fact that the values of $\wp, \wp'$ at these points are weight-two
modular forms with trivial multiplier systems for $\Gamma(2)$.
See Section \ref{secmodularityattorsion} for more details on this.

This family admits furthermore a uniformization via Jacobi $\theta$-functions
compatible with the above Weierstrass $\wp$-uniformization in the sense that the origins
for the group law are the same. See \cite{Dolgachev:1997} for details.
It turns out that the open GW point $[0,-1,1]$ in \eqref{eqnopenGWpoint} is a $3$-torsion point.

\subsubsection{$K_{W\mathbb{P}[1,1,2]}$}
\label{secexKWP}

The affine part of the mirror curve given in Example \ref{ex:3} is equivalent to
\begin{equation}\label{eqnmirrorcurveofKF0}
y^2+{x^{4}}+y+b_{4} x^{2}y+b_{0}xy=0\,.
\end{equation}
The parameters $b_{0},b_{4}$ are related to those in Example \ref{ex:3}
by
$
q_{1}=b_{4}b_{0}^{-2},
q_{2}=b_{0}^{-4}
$.
The rational function $x$ induces a hyperelliptic structure on the mirror curve with generic
$b_{4},b_{0}$.
Another different hyperelliptic structure for the mirror curve is induced from
the equation $y^{2}+1+x^2 y+ b_{4}y+b_{0}xy=0$. The discussion below applies similarly to this case.

The $\wp$-uniformization can be obtained from the algorithm in \cite{Connell1996:elliptic}.
It is accomplished by the following sequence of change of coordinates which induce bi-regular maps on the curves.
First we make the change of coordinates
\begin{equation}
\alpha=2^{2\over 3}\kappa^{2} X -{1\over 12} (b_{0}^{2}+2 b_{4})\,,\quad
\beta=\kappa^{3}  Y-{1\over 2}b_{0}  (\alpha+{1\over 2}b_{4})\,,
\end{equation}
where $\kappa$ is some constant arising from homothety.
Then we set
\begin{equation}
x=\beta^{-1}\left(2^{2\over 3}\kappa^{2} X +{1\over 3} (b_{4}-{1\over 4}b_{0}^{2}) \right)\,,\quad
y=-{1\over 2} +x (\alpha x -{1\over 2}b_{0})-{1\over 2} (1+b_{0}x+b_{4}x^2)\,.
\end{equation}
Then the equation for the curve becomes the Weierstrass normal form
\begin{equation}
Y^{2}=4X^{3}-a  X-b\,,
\end{equation}
with
\begin{equation}
a=  \kappa^{-4}
{(b_{0}^{4}-8b_{0}^2 b_{4}+16 b_{4}^2-48)\over  2^{{1\over 3}} \cdot 24}\,,
\quad
b=- \kappa^{-6} {(b_{0}^2-4b_{4})
	(b_{0}^4-8b_{0}^2 b_{4}+16 b_{4}^2-72)\over 864}
\,.
\end{equation}
The $j$-invariant is
\begin{equation}
j={(b_{0}^4-8b_{0}^2 b_{4}+16 b_{4}^2-48)^{3}\over
	(b_{0}^2-4 b_{4}+8)(b_{0}^2-4 b_{4}-8)}\,.
\end{equation}
From these computations it is easy to see that the parameters $b_{0},b_{4}$ enter
the discriminant and the $j$-invariant through the combination
\begin{equation}
s=(b_{0}^2-4b_{4})^2={ (1-4 q_{1})^{2}\over q_{2}}
\,,
\quad
j={(s-48)^{3} \over s-64}\,.
\end{equation}
We recognize (see for instance \cite{Maier:2009}) that $s$ is a Hauptmodul $t_{2}+64$ for $\Gamma_{0}(2)$.
Up to an $SL_{2}(\mathbb{Z})$-transform, one has
\begin{equation}
t_{2}=64 {   (\theta_{2}^{4}(2\tau) +\theta_{3}^{4} (2\tau))^{2}    \over  \theta_{4}^{8}(2\tau)}-64\,.
\end{equation}
Solving $a=g_{2}, b=g_{3}$, we obtain
\begin{equation}\label{eqntheta42tau}
\kappa=2^{-{1\over 3}} \pi^{-1} \theta_{4}^{-2} (2\tau)\,.
\end{equation}
This is a modular form for $\Gamma_{0}(4)$ with a non-Dirichlet multiplier systems (see for instance \cite{Maier:2009, Maier:2011}).
By passing to the smaller modular subgroup $\Gamma_{0}(4)\cap \Gamma(2)$, it has the Dirichlet multiplier system
$\chi_{-4}$, see \cite{Maier:2011}. By further passing to, say, $\Gamma_{1}(4)\cap \Gamma(2)$, it then has the trivial multiplier system.\\

We now consider the shift $\epsilon$ in
$X=\wp(u+\epsilon), Y=\wp'(u+\epsilon)$.
It is such that the point $[x_{O},y_{O},1]$ is a ramification point for
\eqref{eqnmirrorcurveofKF0}. By completing square, we see that \eqref{eqnmirrorcurveofKF0} is transformed into
\begin{equation}\label{eqnhyperellipticcurvecompletingsquare}
(y+h(x))^{2}=g(x)\,,
\quad
h(x)={1\over 2} (1+b_{0}x+b_{4}x^2)\,,
\quad
g(x)=-x^{4}+h^2(x)\,.
\end{equation}
In particular,  the coordinate for the branch point $x_{O}$
satisfy the equation $g(x_{O})=0$ which for generic parameters $(b_{0},b_{4})$
has four distinct finite solutions.
These four solutions are given by
\begin{equation}\label{eqnbranchpointcoordinatesKF0}
x={  -b_{0}\pm \sqrt{b_{0}^{2}-4 (b_{4}+ 2)} \over 2 (b_{4}+ 2)}\,,
\quad
{  -b_{0}\pm \sqrt{b_{0}^{2}-4 (b_{4}- 2)} \over 2 (b_{4}-2)}\,.
\end{equation}
Recall that $s=(b_{0}^2-4b_{4})^2$ is a modular function $t_{2}+64$ for $\Gamma_{0}(2)$,
we claim that the square roots
$(b_{0}^{2}-4b_{4}-8)^{1\over 2}, (b_{0}^{2}-4b_{4}+8)^{1\over 2}$ are also modular functions, by passing to a smaller modular subgroup $\Gamma<\Gamma_{0}(2)$.
Indeed, from the formulae in \cite{Maier:2009}, we see that
\begin{equation}
(b_{0}^{2}-4b_{4}-8)=t_{4}
\end{equation}
for a Hauptmodul $t_{4}$ for $\Gamma_{0}(4)$.
Up to a $SL_{2}(\mathbb{Z})$ transform on $\tau$, it is given by
\begin{equation}
t_{4}(\tau)=2^{8} {\eta^{8}(4\tau)\over \eta^{8}(\tau)}\,.
\end{equation}
Hence $(b_{0}^{2}-4b_{4}-8)^{1\over 2}$
is a modular form for $\Gamma_{0}(4)$ with a quadratic multiplier system.
Basing on the explicit multiplier system for $\eta(\tau)$ (given in e.g.  \cite{Maier:2011}) one can in fact prove that
it is a modular function for $\Gamma_{1}(8)$ with trivial multiplier system.
We also have
$t_4+16=(b_{0}^{2}-4b_{4}+8)=(t_{8}+4)^{2}$ for a certain Hauptmodul $t_{8}$ for the modular group
$\Gamma_{0}(8)$, see again \cite{Maier:2009} for the details. That is, $(b_{0}^{2}-4b_{4}+8)^{1\over 2}=t_8+4$
is modular function for $\Gamma_{0}(8)$.

Therefore by passing to the smaller modular subgroup $\Gamma_{1}(8)$,  the roots
do not create trouble in discussing modularity.
Furthermore, by making use of a $\theta$-uniformization similar to the $K_{\mathbb{P}^{2}}$ case, we see that the open GW point \eqref{eqnopenGWpoint} is a $4$-torsion point.\\

One can obtain interesting non-trivial one-parameter families by restricting to
one-dimensional subspaces in the $(b_{0}, b_{4})$-space with non-constant $b_{0}^2-4b_{4}$
which determines the complex structure through the $j$-invariant above.
For example, by restricting to $b_{4}=0$, we get an one-parameter family parametrized by $b_{0}$
such that $b_{0}^{4}$ is the Hauptmodul $t_{2}+64$ for $\Gamma_{0}(2)$.
This corresponds to the one-parameter family
\begin{equation}
(q_{1},q_{2})=(0, s)\,,
\quad
s=(t_{2}+64)^{-1}\,.
\end{equation}
Hence indeed as discussed in Section \ref{sec:one-parameter}, after the restriction both $b_{0},b_{4}$
become modular functions for  the modular group $\Gamma_{0}(2)$ determined by the subfamily.
In particular, $s$ is a modular function for the subgroup $\Gamma=\Gamma_{1}(8)$.

By passing to the intersection with $\Gamma_{1}(4)\cap\Gamma(2)\cap\Gamma_{1}(8)=\Gamma(2)\cap \Gamma$
which incorporates the multiplier system for $\kappa$ in the uniformization and the issue on roots of modular forms above,
we see that the value of $x$ and hence of $y=-h(x)$
at any ramification point are modular forms for  $\Gamma(2)\cap \Gamma$.

\subsubsection{$K_{\mathbb{P}^{1}\times \mathbb{P}^{1}}$}
\label{P1timesP1Mirrorcurveuniformization}

Then affine part of the mirror curve given in Example \ref{ex:2} is equivalent to
\begin{eqnarray}
y^{2}+(1+x+q_{1}x^2) y+q_{2}x^{2}=0\,.
\end{eqnarray}
We follow the algorithm in \cite{Connell1996:elliptic} to reduce it to the Weierstrass normal form.
This is accomplished by the following sequence of change of coordinates
which induce bi-regular maps on the curves.
First we set
\begin{equation}
\alpha= 2^{2\over 3}\kappa^{2} X+{1\over 12} (-1-2q_{1}+4q_{2})\,,\quad
\beta= \kappa^{3} Y- {1\over 2} (\alpha+{1\over 2}q_{1})\,,
\end{equation}
where again $\kappa$ is some constant arising from homothety.
Then we make a change of coordinates
\begin{equation}
x=   \beta^{-1}\left(2^{2\over 3}\kappa^{2} X +{1\over 6} (1+2q_{1}-4q_{2})\right)\,,\quad
y=-{1\over 2}+ x (\alpha x -{1\over 2})  - {1\over 2} (1+x +q_{1} x^2)\,.
\end{equation}
Then the equation for the curve becomes the Weierstrass normal form
\begin{equation}
Y^{2}=4X^{3}-a X-b
\end{equation}
with
\begin{eqnarray}
a&=&2^{-{1\over 3}} 24^{-1} \kappa^{-4}
( 16 q_{1}^2-16 q_{1} q_{2} +16 q_{2}^2 - 8 q_{1}-8q_{2}+1)
\,,\nonumber \\
b&=& 864^{-1}\kappa^{-6}
(4q_{1}+4q_{2}-1)
(16 q_{1}^2-40 q_{1}q_{2} +16 q_{2}^2 - 8 q_{1}-8q_{2}+1)\,.
\end{eqnarray}
The $j$-invariant is
\begin{eqnarray}
j&=&
{ ( 16 q_{1}^2-16 q_{1}q_{2} +16 q_{2}^2 - 8 q_{1}-8q_{2}+1)^3
	\over
	q_{1}^2 q_{2}^2 ( 16 q_{1}^2-32 q_{1} q_{2} +16 q_{2}^2 - 8 q_{1}-8q_{2}+1)}\,.
\end{eqnarray}
From this it is easy to see that the parameters $q_{1},q_{2}$ determine the complex
structure of the curve through
\begin{equation}
s= 16 {q_{1}^2+q_{2}^2\over q_{1}q_{2}}- 8 {q_{1}+q_{2}\over q_{1}q_{2}}+{1\over q_{1}q_{2}}\,,
\quad
j={ (s-16)^3\over s-32}\,.
\end{equation}
We recognize that $s$ is a Hauptmodul for $\Gamma_{0}(2)$.\\

One can obtain one-parameter subfamilies by restrictions to one-dimensional spaces with non-constant $j$.
For example,
	taking $q_{1}=q_{2}=s$, we have
	\begin{eqnarray}
	j(s)&=&
	{(1-16 s+16s^2)^3\over s^4 (1-16s)} \,.
	\end{eqnarray}
	We recognize that $s$ is the Hauptmodul $-1/t_{4}$ for $\Gamma_{0}(4)$, see e.g. \cite{Maier:2009} for details.
	One can then solve for $\kappa$ to be
	\begin{equation}\label{eqnkappafromGamma04}
	\kappa=2^{-{7\over 3}}\pi^{-1}\theta_{2}^{-2}(2\tau)\,.
	\end{equation}
	From  \cite{Maier:2011} we know this is a modular form for $\Gamma_{0}(4)\cap \Gamma(2)$
	with the Dirichlet character $\chi_{-4}$.
	By passing to the subgroup  $\Gamma_{1}(4)\cap \Gamma(2)$, we see that $\kappa$ is a modular form with trivial multiplier system.
	A similar computation as in the previous cases by using $\theta$-uniformization shows that the open GW point \eqref{eqnopenGWpoint} is an $8$-torsion.
Hence indeed as discussed in Section \ref{sec:one-parameter}, after the restriction both $q_{1},q_{2}$
become modular functions for a certain modular group $\Gamma$ depending on the subfamily.

Similar discussions in Section \ref{secexKWP} on the shift $\epsilon$ and on values of
$x,y$ at ramification points apply. Namely,
the values of $x$ and hence of $y=-h(x)$
at any ramification point are modular forms for $\Gamma(2)\cap\Gamma$.
We omit the tedious computations here.

\subsubsection{$K_{\mathbb{F}_{1}}$}

The affine part of the mirror curve given in Example \ref{ex:4} is equivalent to
\begin{eqnarray}\label{eqnKF1mirror1}
y^2+y+xy+q_1x+q_2x^2y=0\,.
\end{eqnarray}
For this hyperelliptic structure, we
first apply
\begin{eqnarray*}
\alpha=4^{1\over 3}\kappa^{2}X-{1\over 3}({1\over 4}+{1\over 2}q_2)\,,\quad
\beta=\kappa^{3}Y-\left(({1\over 2}-q_1)\alpha+{1\over 4}{q_2}\right)\,,
\end{eqnarray*}
where $\kappa$ is an undetermined constant arising from homothety.
Then we set
\begin{eqnarray*}
x=\beta^{-1}{(\alpha+{q_2\over 2}-q_1^2+q_1)}\,,\quad
y=-{1\over 2}(1+x+q_2x^2)-{1\over 2}+x(x\alpha-({1\over 2}-q_1))\,.
\end{eqnarray*}
Then equation \eqref{eqnKF1mirror1}  is transformed to the Weierstrass normal form
\begin{eqnarray}
Y^2=4X^3-a X- b
\end{eqnarray}
with
\begin{eqnarray}
a&=&2^{-{1\over 3}} 24^{-1}\kappa^{-4} \left( (1-4q_2)^2+24q_1q_2\right)\,,\\
b&=& 864^{-1}\kappa^{-6}\left( (4q_2-1)^3+36q_1q_2(4q_2-1)-216q_1^2q_2^2 \right)\,.
\end{eqnarray}
The $j$-invariant is given by
\begin{eqnarray}
j=-{(1 -8 q_2 +24  q_1 q_2 + 16 q_2^2)^3\over
 q_1^2 q_2^3 \left(q_1 - (1 - 4 q_2)^2 - 36 q_1 q_2 + 27 q_1^2 q_2\right)}\,.
\end{eqnarray}

We can obtain interesting subfamilies by restricting the above two-parameter family to one-dimensional ones.

\begin{itemize}

	\item
First make the change of parameters $q_1=\tilde{q_1} \tilde{q_2}^{-1},q_2=\tilde{q_2}$.
Then taking $\tilde{q}_{2}=0, \tilde{q}_{1}=s$, we obtain
	\begin{equation}
	j=- {(1+24s)^{3}\over s^{3} (1+27 s)}\,.
	\end{equation}
	We recognize that $s$ is a Hauptmodul for $\Gamma_{0}(3)$, see e.g. \cite{Maier:2009} for details.
	This is consistent with the observation that setting $\tilde{q}_{1}=0$ in \eqref{eqnKF1mirror1}
	reduces the mirror curve of $K_{\mathbb{F}_{1}}$
	to one which is isomorphic to the mirror curve of $K_{\mathbb{P}^{2}}$ after a change of coordinates corresponding to a biregular morphism. In fact, this amounts to the
	restriction from the set of lattice points in the defining polytope of the former to that of the latter.
	In particular, the open GW point \eqref{eqnopenGWpoint} is a 3-torsion in the new coordinates.

	\item

	Taking $q_{1}=1, q_{2}=s$, we obtain
	\begin{eqnarray}
	j={(16s^2+16s+1)^3\over s^4(16 s+1)}\,.
	\end{eqnarray}
	We recognize that $s$ is the Hauptmodul $1/t_{4}$ for $\Gamma_{0}(4)$.
	Choosing $s$ to be
	\begin{equation}
	s=2^{-8}{\eta(\tau)^8\over \eta(4\tau)^8}\,,
	\end{equation}
	one can solve for $\kappa$ to be
\begin{equation}
\kappa=2^{-{13\over 3}} \pi^{-1}{\eta(2\tau)^2\over \eta(4\tau)^4}\,.
\end{equation}
This is proportional to the quantity in \eqref{eqnkappafromGamma04}
and the subtlety on the multiplier systems 	can be resolved by passing to a smaller modular subgroup in the same way.
	Deriving the $u$-coordinate for the open GW point is more complicated in this case.
\end{itemize}

Indeed in these examples, as discussed in Section \ref{sec:one-parameter}, after the restriction both $q_{1},q_{2}$
become modular functions for a certain modular group $\Gamma$ depending on the subfamily.

Similar to Section
\ref{secexKWP},
by passing to the smaller modular subgroup
if needed, we see that
the values of $x$ and hence of $y=-h(x)$
at any ramification point are modular forms for $\Gamma(2)\cap\Gamma$.
Again we omit the tedious computations here.

\begin{rem}
Another hyperelliptic structure is
\begin{eqnarray}
y^2+(1+x +q_{1}x^2)y+ q_2 x^3=0\,.
\end{eqnarray}
The underlying algebraic curves are bi-regular, with the bi-regular map easily identified from the relations to the toric characters in
\eqref{eqn:PDelta}.
The $\wp$-uniformization is again derived from the algorithm in \cite{Connell1996:elliptic}.
The details are as follows.
We first make the change of variables
\begin{equation}
\alpha= 2^{2\over 3} \kappa^{2} X-{1\over 12}(2 q_{1}+1)\,,\quad
\beta=  \kappa^{3} Y-{1\over 2} (\alpha+{1\over 2} q_{1}-q_{2})\,.
\end{equation}
Then we set
\begin{equation}
x= \beta^{-1} \left(   \alpha+{1\over 2} q_{1} \right)\,,\quad
y=-{1\over 2}+ x    (\alpha x -{1\over 2})  +
{1\over 2} (1+x+q_{1}x^{2})\,.
\end{equation}
The Weierstrass normal form is the same as the one for the first hyperelliptic structure as it should be.
The different hyperelliptic structures have different ramification data and open GW points.
One can consider the special one-parameter sub-families as above.
The discussion in Section \ref{secexKWP} on the values of $x,y$ at the ramification points
also applies here.

\end{rem}

\begin{rem}
	Invoking the correspondence between the linear relations in the homogeneous quotient
	construction of toric variety and the Mori cone of curves in the toric variety, we see that the above specializations
	correspond to different walls in the second fan, which models the moduli space of K\"ahler structures of the A-model.
	Hence topological recursion, when combined with the modularity studied in this work, provides a promising tool in studying the phase transition and wall crossing phenomena, along the lines in e.g. \cite{Witten:1993, Chiang:1999tz, Alim:2008kp}.
	We hope to return to this in a future work.

\end{rem}

\section{Proof of main theorems}
\label{sec:proof}

In this section we prove the main theorems for  the examples
$\cX=K_S$ for $S=\bP^2, \bP^1\times \bP^1, W\bP[1,1,2], \mathbb F_1$.
We will start from a general discussion on the modularity of
the differentials $\{\omega_{g,n}\}_{g,n}$ produced from applying topological recursion to
a genus one mirror curve $C$
whose affine part\footnote{Only the affine part of the curve is relevant in topological recursion.} is given by  \eqref{eqnHxy=0} with hyperelliptic structure given by $x$.

We shall only focus on one-parameter subfamilies.
	However, many of the results for the one-parameter subfamilies, such as the
	 structure for the ring in Theorem \ref{thmhighergenusWgn}
	 and the holomorphic anomaly equations in Theorem \ref{thmhae}
	can be easily generalized to topological recursion for the full multi-parameter families. The only difference
	is the lack of a better understanding on the moduli space interpretation of the rest of the parameters (other than the complex structure
	modulus) from the view point of the mirror curve.

\subsection{Expansions of basic ingredients in topological recursion}

\subsubsection{Local coordinates for expansions}
\label{seclocaluniformizer}

We use $[x_1,x_2,x_3]$ to denote a point on the (compactified) mirror curve $C$,
which are the first three homogeneous coordinates of $\bP^{\fp+2}$
in \eqref{eqn:PDelta} -- namely $x=x_1/x_3$ and $y=x_2/x_3$.
For a generic mirror curve, the set $R^{\circ}$ of finite (i.e., in the $x_{3}=1$ patch) ramification points is a subset of the affine mirror curve $C^{\circ}$.

In Section \ref{secallfourexamples}, we have
made the choice of origin for the group law for the mirror curve.
For $\cX=K_{\bP^2}$, the shift $\epsilon$ in uniformization formula has chosen to be zero.
Accordingly, we have
$R^{\circ}=\{u={1\over 2}\,,
{\tau\over 2}\,,
{1+\tau\over 2}\}
$.
For the other three cases $\cX=K_{\bP^1\times \bP^1},  K_{ W\bP[1,1,2]}, \mathbb F_1$, we have
$R^{\circ}=\{u=0\,, {1\over 2}\,,
{\tau\over 2}\,,
{1+\tau\over 2}\}$.
According to Part $2$ of Lemma \ref{lemramificationofhyperellipticgenusonecurves}, the hyperelliptic involution $*$ on the mirror curve is induced by the involution $u\mapsto -u$ on the Jacobian. We also use $*$ to denote the induced actions on functions and differentials.

We need the notion of local uniformizer  for the calculus on the mirror curve $C$.
In what follows, we always use the local uniformizer\footnote{The should not be confused with the K\"ahler parameter discussed earlier.}
\begin{equation}
T=u-u(p)
\end{equation}
near a point $p$ corresponding to $u(p)$ under uniformization.
We shall also identify a point $p\in C$ with its $u$-coordinate which is defined modulo translation by elements in
the lattice
$\mathbb{Z}\oplus \tau \mathbb{Z}$ mentioned in Section \ref{secuniformizationofellipticcurves}.

\subsubsection{The $\log$-differential and Bergmann/Schiffer kernel}
\label{seclogBergmanSchiffer}

The basic ingredients in Eynard-Orantin topological recursion
are the $\log$-differential\footnote{The differential $\lambda$, which involves logarithm, is derived as the dimension reduction of the Calabi-Yau form of the non-compact CY 3-fold \cite{Chiang:1999tz, Aganagic:2000gs, Aganagic:2002} and relates to mirror symmetry. Its rigorous definition uses mixed Hodge structure \cite{Batyrev:1993variations, Stienstra:1997resonant, Konishi:2010local}. In the current genus one case, we understand the logarithm via the formal group of the elliptic curve \cite{Silverman:2009arithmetic}. In the literature, sometimes another version $\lambda=ydx$ is used. While much easier to deal with, $ydx$ is not directly related to toric CY $3$-folds  by mirror symmetry.}
 $\lambda=\log y \cdot dx/x$ and the Bergman kernel $B$.
The differential $\lambda$
depends on the choice of the local coordinates $x,y$ as displayed in \eqref{eqnHxy=0}.\\

Instead of the Bergmann kernel $B$ in \cite{Eynard:2007invariants} (which produces
differentials $\{\omega_{g,n}\}_{g,n}$) we usually work with the Schiffer kernel $S $  (which produces
differentials $\{\widehat \omega_{g,n}\}_{g,n}$).
The Schiffer kernel is independent of the Torelli marking, as defined in Section \ref{sec:remodeling}.

In the genus one case, the Schiffer kernel is given by
\begin{equation}\label{eqngenusoneSchifferkernel}
S(u_{1},u_{2})
=(\wp(u_{1}-u_{2})+\widehat{\eta}_{1})du_{1}\boxtimes du_{2}\,,
\quad
\widehat{\eta}_{1}=2\zeta(2)\widehat{E}_{2}={\pi^{2}\over 3}(E_{2}+{-3\over \pi \mathrm{Im}\tau})\,.
\end{equation}
Here although the quantity $\tau$ depends on the Torelli marking,
the Schiffer kernel $S$ does  not. An advantage, besides being modular, is that it keeps track of part of the combinatorics in topological recursion through the non-holomorphic dependence in $\tau$. This will be used later in the discussion of holomorphic anomaly equations in Section \ref{secholomorphicanomalyequations}.\\

Through this work, we are only interested in the coefficient part of the differential $\omega_{g,n}$
with respect to the trivialization
$du_{1}\boxtimes du_{2}\cdots \boxtimes du_{n}$, constructed from topological recursion.
By abuse of terminology, we say $\omega_{g,n}$ has modular properties (like being Jacobi forms)
if its coefficient has so.
Hence the Schiffer kernel $S$ is regard as an almost-meromorphic Jacobi form according to Definition \ref{dfnquasialmostmeromorphicJacobiforms}.
Similarly, the Bergmann kernel $B$ is quasi-meromorphic Jacobi form.

\subsubsection{Modularity of Taylor coefficients of Jacobi forms at torsion points}
\label{secmodularityattorsion}

The following result proves to be useful in discussing modularity of Taylor coefficients of
meromorphic Jacobi forms \cite{Eichler:1984}.
Suppose $\Phi$ is a meromorphic Jacobi form of weight $m$, then its $k$th Taylor coefficient
at $x_{0}+y_{0} \tau$ is a meromorphic modular form of weight $m+k$
for the modular group consisting of matrices $\gamma\in SL_{2}(\mathbb{Z})$
such that $\gamma (x_{0}+y_{0} \tau)=x_{0}+y_{0} \tau\, \mathrm{mod}\, \mathbb{Z}\oplus \mathbb{Z}\tau $.
See \cite{Dolgachev:1997} for a nice exposition of these facts.

Consider the case $\Phi=\wp$
which is a  meromorphic Jacobi form of  of weight $2$ with level $ SL_{2}(\mathbb{Z})$.
At the $2$-torsion points, the modular group can be taken to be $\Gamma(2)$.
The same statement is true for the meromorphic Jacobi form $\wp'$, and higher derivatives of $\wp$.
In the higher derivative cases, we can alternatively use the algebraic relation $(\wp')^2=4\wp^3-g_{2}\wp-g_{3}$
 satisfied by $\wp$ and $\wp'$ in \eqref{eqnalgebraicrelationwpwp'} and then apply induction.
This when combined with Lemma
Lemma \ref{lemramificationofhyperellipticgenusonecurves} and Lemma
\ref{lemmodularityofvaluesatramificationpoints}
would imply that the differentials produced by topological recursion are quasi- or almost- meromorphic Jacobi forms, as we shall see below.\\

For later use, we recall the values of $\wp$
\begin{eqnarray}\label{eqnperiodsof2ndkind}
e_{1}&:=&\wp({1\over 2})=2\zeta(2) (\theta_{3}^4+\theta_{4}^4)\,,\nonumber\\
e_{2}&:=&\wp({\tau\over 2})=2\zeta(2) (-\theta_{2}^4-\theta_{3}^4)\,,\nonumber\\
e_{3}&:=&\wp({1+\tau\over 2})=2\zeta(2) (\theta_{2}^4-\theta_{4}^4)\,.
\end{eqnarray}
See \cite{Zagier:2008} for the convention of the $\theta$-constants above.
As explained earlier in Section \ref{secmodularityattorsion}, these are modular forms for $\Gamma(2)$ with trivial multiplier systems.
We also denote
\begin{equation}\label{eqnnonholomorphicperiodsof2ndkind}
\widehat{e}_{k}:=e_{k}+\widehat{\eta}_{1}\,,
\quad
k=1,2,3\,,
\quad
\widehat{e}_{0}:=\widehat{\eta}_{1}\,.
\end{equation}
The following Laurent expansion of $\wp$ at $u=0$ is also useful
\begin{equation}\label{eqnLaurentexpansionofwp}
\wp(u)={1\over u^2}+\sum_{k=1}^{\infty} (2k+1) 2\zeta_{2k+2}E_{2k+2}  u^{2k}\,,
\end{equation}
where $\zeta_{2k+2}$ is the $\zeta$-value and $E_{2k+2}$
is the Eisenstein series of weight $2k+2$ with normalized leading term in the Fourier expansion to be $1$.

\subsubsection{Local expansions near the ramification points}

In topological recursion one needs to study residues of quantities around ramification points of $x: C\rightarrow \mathbb{P}^{1}$ which gets identified with the group of $2$-torsion points, according to Lemma  \ref{lemramificationofhyperellipticgenusonecurves}.\\

For later use, we now study $\lambda-\lambda^{*}$ around the ramification points in $R^{\circ}$. Note that vanishing locus of $y$ is away from $R^{\circ}$, hence $\log y$ is single-valued if we fix a branch of logarithm once and for all. We shall choose the principal branch which takes the value $0$ when $y=1$.

We simplify $\lambda-\lambda^{*}$ by making use of the results on uniformization as follows.
We know for an one-parameter subfamily, $x,y$
are rational functions in $\wp(u+\epsilon),\wp'(u+\epsilon)$ for some shift $\epsilon$,
with coefficients lying in the fractional field  $\mathcal{M}(\Gamma)$ of the ring $M(\Gamma)$ of modular forms for some modular group $\Gamma$ depending on the curve family $\mathcal{C}$.
Under the involution $*$ the rational function $x$
is fixed while for $y$ we have
\begin{equation}
y={y+y^{*}\over 2}+{y-y^{*}\over 2}\,,
\quad
y^{*}={y+y^{*}\over 2}-{y-y^{*}\over 2}\,,
\end{equation}
Furthermore since $y\neq 0$ at a ramification point in $R^{\circ}$ where $y-y^{*}=0$,
we know $y+y^{*}$ is not vanishing at a ramification point in $R^{\circ}$.
We then have
\begin{eqnarray}
\lambda-\lambda^{*}&=&
\log  y {dx\over x} -
\log  y^{*}{dx^{*}\over x^{*}}
=\log \left( { {y+y^{*}\over 2} +{y-y^{*}\over 2}\over {y+y^{*}\over 2}-{y-y^{*}\over 2} }\right) {dx\over x}\,.
\end{eqnarray}
At a finite ramification point we also have $x\neq 0, dx=0$, we then define
\begin{eqnarray}\label{eqndifferenceoflambda}
 \Lambda&:=&2 \sum_{k=0}^{\infty} {1\over 2k+1} ({y-y^{*}\over y+y^{*}})^{2k+1}\partial_{u}x {du\over x}\,,
\end{eqnarray}
which is an expression for $(\lambda-\lambda^{*})$ near each ramification point.
The vanishing order of $y-y^{*}$ at the ramification point is $1$ since the curve $C$ is smooth.
According to the results on uniformization $ \Lambda$ is a meromorphic Jacobi form, its weight
is $1$ coming from the $dx/x$ part: the coefficient part has weight zero.

We can further expand the above expression
\eqref{eqndifferenceoflambda}
in terms of the local uniformizing parameter $T=u-u_{r}$, where $u_r$ is the $u$-coordinate of the ramification point $r\in R^{\circ}$.
Then we have
\begin{equation}
\wp(u+\epsilon)=\wp(T+u_{r}+\epsilon)\,.
\end{equation}
When $u_{r}+\epsilon=0$ modulo $\mathbb{Z}\oplus \tau\mathbb{Z}$, the Laurent expansion of
$\wp(u+\epsilon), \wp'(u+\epsilon)$ in the local uniformizer $T$ follow from \eqref{eqnLaurentexpansionofwp}.
Otherwise we have the Taylor expansion
\begin{equation}\label{eqnchangebetweenalgebraicandtranscendental}
\wp(u+\epsilon)=\left(\sum_{k=0}^{\infty}{T^{k}\over k!}\wp^{(k)}(u_{r}+\epsilon)  \right)
\,.
\end{equation}

We can also expand the Schiffer kernel \eqref{eqngenusoneSchifferkernel}
around a ramification point $r\in R^{\circ}$ with respect to one of its arguments.
The expansion in $T=u-u_r$ is
\begin{equation}\label{eqnexpansionforSchifferkernel}
S(u,v)=(\wp(T+u_{r}-v)+\widehat{\eta_{1}})dT\boxtimes dv
=
\left(\sum_{k=0}^{\infty}{T^{k}\over k!}(\wp^{(k)}(u_{r}-v)+\widehat{\eta_{1}}^{(k)}) \right)dT\boxtimes dv\,.
\end{equation}
One has $\widehat{\eta_{1}}^{(k)}=0$ unless $k=0$ in which case
$\widehat{\eta_{1}}^{(k)}=\widehat{\eta_{1}}$.

\subsection{Modular properties of $\{\omega_{g,n}\}_{g,n}$ and ring structure}
The differentials  $\widehat \omega_{g,I+1}, 2g-2+I+1>0$
are constructed recursively in \cite{Eynard:2007invariants} through
\begin{equation}
\begin{split}
\label{eqnWgI+1recursion}
\widehat \omega_{g,I+1} (u_{0}, u_{I})
&=\sum_{r\in R^{\circ}} Res_{v=r} \,
K(u_{0},v)\cdot \\
&\left[ \widehat \omega_{g-1,I+2} (v,v^{*}, u_{I})+
\sum'_{\substack{g_{1}, g_{2}\\ g=g_{1}+g_{2}}}\sum'_{\substack{J, K\\ I=J\sqcup K}} \widehat \omega_{g_{1}, J+1} (v,u_{J})\cdot \widehat \omega_{g_{2}, K+1} (v^{*},u_{K})
\right]
\,.
\end{split}
\end{equation}
Here the notation $\sum'$ means that the range in the sum
is such that the construction is strictly recursive.
We have also used the notations $I,J,K$ to denote the sets of indices and the corresponding cardinalities.
The quantity $ \hat{F}_{g}=\widehat{\omega}_{g,0}, g\geq 2$, called genus $g$ free energy, is defined in \cite{Eynard:2007invariants} through
 \begin{equation}\label{eqnFgrecursion}
 \hat{F}_{g}:={1\over (2-2g)}\sum_{r\in R^{\circ}}Res_{r} (d^{-1}\lambda \cdot \widehat{\omega}_{g,1})\,.
 \end{equation}

In the above constructions \eqref{eqnWgI+1recursion} and \eqref{eqnFgrecursion}, the quantity $K$
is the recursion kernel  \cite{Eynard:2007invariants} defined by
\begin{equation}\label{eqnrecursionkernel}
K(u,v)= {d^{-1}S \over\lambda(v)-\lambda(v^{*})}= {d^{-1}S \over\lambda(v)-\lambda^{*}(v)},
\end{equation}
where
\begin{equation}\label{eqnnumeratorofrecursionkernel}
d^{-1}S:={1\over 2} \int_{2u_{r}+v^*}^v S(u,\bullet)\,.
\end{equation}
Again we understand the logarithm in the denominator of $K$ from the formal group point of view \cite{Silverman:2009arithmetic} as before.
This means that both \eqref{eqnrecursionkernel} and \eqref{eqnnumeratorofrecursionkernel}
are expressed in terms of Laurent series in the local uniformization $T=v-u_{r}$ near a ramification point $u_{r}\in R^{\circ}$.
The shift $2u_{r}$ in the lower bound $2u_{r}+v^{*}=2u_{r}-v$ in \eqref{eqnnumeratorofrecursionkernel} is needed
such that $d^{-1}S$ vanishes at the ramification point $v=u_{r}$, i.e., $T=0$.
The quantity $d^{-1}\lambda$ in \eqref{eqnFgrecursion} is defined in a similar way
such that $2(d^{-1}\lambda)' (v)=\lambda(v)-\lambda^{*}(v) $.\\

The differentials $\widehat \omega_{g,n}, 2g-2+n\leq  0$, that is
$(g,n)=(0,1),(0,2),(1,0)$,
 are dealt with separately below.
For the $(g,n)=(0,1)$ case, the differential $\widehat \omega_{0,1}$ is defined in \cite{Eynard:2007invariants} to be zero.

\subsubsection{Disk potential}
The mirror counterpart of the superpotential $W$ is a primitive  \cite{Aganagic:2000gs, Aganagic:2002} of the differential $\lambda$,
integrated along a certain chain on the curve
$C$.
By definition, its derivative $\partial_{x}W$, called the disk potential, satisfies
\begin{equation}
{\partial W\over \partial x }=\lambda=\log y\cdot  {1\over x}\,.
\end{equation}
We arrive at the following result.
\begin{prop}\label{propdiskpotentialJacobi}
The disk potential $\partial_{x}W$
is the logarithm of a meromorphic Jacobi form whose modular group $\Gamma$
	is determined by the one-parameter subfamily of the mirror curve family $\mathcal{C}$.
\end{prop}

\subsubsection{Annulus potential}

The differential $\omega_{0,2}$
is mirror to the annulus amplitude. It is defined to be the Bergmann kernel $B$
and is the holomorphic limit of the Schiffer kernel $\widehat{\omega}_{0,2}:=S$. It is a quasi-meromorphic Jacobi form.

The quantity $d^{-1}S$ is a "formal" almost-meromorphic Jacobi form of "formal" weight $1$
in the sense that its derivative (in $v$) is an almost-meromorphic Jacobi form of weight $2$.
The recursion kernel $K$, as the quotient of
$d^{-1}S$ by the Jacobi form in \eqref{eqndifferenceoflambda} is also regarded as a "formal"  almost-meromorphic Jacobi form.

\begin{prop}\label{lemannuluspotentialJacobi}
The annulus  amplitude $ \omega_{0,2}=B$
is a weight $2$, index $0$, level $\Gamma(1)$, quasi-meromorphic Jacobi form. It is symmetric in its arguments.
The recursion kernel $K=d^{-1}\widehat \omega_{0,2}/(\lambda-\lambda^{*})$ is a formal almost-meromorphic Jacobi form of formal weight $0$.
\end{prop}

\subsubsection{Higher genus modularity}

We will use topological recursion to prove the modularity of $\{\widehat \omega_{g,n}\}_{g,n}$ for higher $(g,n)$.

\subsubsection*{Genus one closed case}

The quantity $\widehat \omega_{1,0}=\hat{F}_{1}$, called genus one free energy, involves
the Bergmann $\tau$-function $\tau_{B}$ \cite{Eynard:2007invariants}.
In the current genus one case,
the Bergman $\tau$-function, as an analytic invariant, is given by
\cite{Kokotov:2003bergmann, Kokotov:2004tau, Kokotov:2004tau2}
\begin{equation}
\tau_{B}=\eta^{2}(\tau)\,.
\end{equation}
The genus one free energy
$\hat F_{1}$ is then defined to be
\begin{equation}\label{eqnnonholomorphicgenusonefreeenergy}
\hat F_{1}=-{1\over 2}\ln \tau_{B}-{1\over 12}\ln \prod_{r\in R^{\circ}} {dy\over d(x-x(r))^{1\over 2}}|_{r}-\ln \det Y\,,
\quad
Y=-\pi /\mathrm{Im}\tau\,.
\end{equation}
The second term can be computed to be the logarithm of a modular function.
Taking the holomorphic limit (setting $\mathrm{Im} \tau \to \infty$), we define $dF_{1}
:=\lim_{\mathrm{Im}\tau\to \infty} d\hat F_{1}$. It is shown that in \cite[Theorem 7.9]{FLZ16} that $dF_1^\cX=dF_{1}$. We therefore arrive at the following result.
\begin{thm}
Up to addition by a constant, the genus one closed GW potential $F^\cX_1(\tau)$ is the logarithm of a meromorphic modular form
whose modular group $\Gamma$
	is determined by the one-parameter subfamily of the mirror curve family $\mathcal{C}$.
\end{thm}

\subsubsection*{Higher genera}

Note that in higher genus recursion for $\{\widehat{\omega}_{g,n}\}_{g,n}$,
the disk potential $W$ and genus one free energy $\widehat{F}_{1}$ do not enter, hence
no logarithms of almost-meromorphic Jacobi forms will appear.

We define the total weight of
$\widehat{\omega}_{g,n}(u_{1}, \cdots  , u_{n})$
to be the total wight of its coefficient
with respect to the trivialization $du_{1}\boxtimes \cdots du_{n}$, which
will be proved to be a multi-variable almost-meromorphic Jacobi form as defined in Definition \ref{dfnmultiJacobiform} in Section
\ref{sec:jacobi}.

\begin{thm}
\label{thmhighergenusWgn}
The following statements hold for $\widehat \omega_{g,n}$ with $2g-2+n> 0$.

\begin{enumerate}
\item
The differential $\widehat \omega_{g,n}(u_1,\cdots, u_n), n\neq 0$ is symmetric in its arguments.
In each argument, it only has poles at the ramification points in $R^{\circ}$.
At any of the ramification points, the order of pole in any argument is at most $6g+2n-4$.
Furthermore, the sum of orders of poles over all arguments
in each term in $\widehat \omega_{g,n}(u_1,\cdots, u_n)$ is at most $6g+4n-6$.

\item The differential $\widehat \omega_{g,n}(u_1,\cdots, u_n),n\neq 0$ is a differential polynomial in $S(u_k-u_r),$ $k=1,2,\cdots n\,, r\in R$.
The coefficients of $\widehat\omega_{g,n}$ regarded as a differential polynomial in $S(u_{k}-u_{r}), k=1,2,\cdots n, r\in R$ are elements in the ring
\begin{equation}
\label{eqncoefficientring}
\widehat{\mathcal{K}}:=\mathcal{M}(\Gamma(2)\cap\Gamma) \otimes \mathbb{C}[\widehat{e}_{1}  , \widehat{e}_{2} , \widehat{e}_{3}  ,\widehat{ \eta}_{1} ]=\mathcal M(\Gamma(2)\cap\Gamma)\otimes \bC[\hat E_2]\,.
\end{equation}
In particular, the coefficient of $\widehat \omega_{g,n}(u_1,\cdots, u_n),n\neq 0$ under the trivialization $du_1\boxtimes \dots \boxtimes du_n$
is an almost-meromorphic multi-Jacobi form for the modular group $\Gamma(2)\cap \Gamma$, where the modular group $\Gamma$
	is determined by the one-parameter subfamily of the mirror curve family $\mathcal{C}$.
Its total weight is $n$.
\item The quantity $\hat{F}_{g} ,g\geq 2$ is an almost-meromorpic modular form of weight zero, lying in the ring $\widehat{\mathcal{K}}$ given in  \eqref{eqncoefficientring}.
\end{enumerate}
\end{thm}
\begin{proof}
\begin{enumerate}
\item
The proofs of the first two statements follow by induction basing on the recursion formula  \eqref{eqnWgI+1recursion}, as in \cite{Eynard:2007invariants}.

For the third statement, denote by $N_{g,I}$ the maximum of the order of pole among all arguments and all ramification points in
$\widehat{\omega}_{g,I+1}$ for any $g,I$ not necessarily satisfying the condition $2g-2+(I+1)>0$.
By induction it is easy to show that
\begin{equation}\label{eqnrecursionoforderofpole}
N_{g,I}+2\leq \max_{g_{1},g_{2}, J, K}\{(N_{g_{1}, J} +2)+(N_{g_{2}, K} +2)\} \,,
\end{equation}
where the maximum is taken over all possible partitions of $g$ and $I$.
Direct computations for the first few $(g,n)$'s show that
$N_{0,1}=0, N_{0,2}=2, N_{1,0}=4$.
The estimate \eqref{eqnrecursionoforderofpole} and the initial values imply that
$N_{g,I}\leq 6g+2I-2$ when $2g-2+(I+1)>0$.

For the last statement,
 denote similarly by $\widetilde{N}_{g,I}$ the maximum of the sum of orders of pole over all arguments
in $\widehat{\omega}_{g,I+1}$, for any $g,I$ not necessarily satisfying the condition $2g-2+(I+1)>0$.
 Again by induction we see that
 \begin{equation}\label{eqnrecursionofsumofordersofpole}
 \widetilde{N}_{g,I}+2\leq \max_{g_{1},g_{2}, J, K}\{(\widetilde{N}_{g_{1}, J} +2)+(\widetilde{N}_{g_{2}, K} +2)\} \,.
 \end{equation}
 Direct computation shows that $\widetilde{N}_{0,1}=2, \widetilde{N}_{0,2}=6, \widetilde{N}_{1,0}=4$. The estimate \eqref{eqnrecursionofsumofordersofpole} and the initial values imply that
 $\widetilde{N}_{g,I}\leq 6g+4I-2$ when $2g-2+(I+1)>0$.

\item
We again prove by induction.
Near the ramification point $u_{r}$, we choose the local parameter $T=v-u_{r}$ in order to evaluate the residues.

We first consider the genus zero case. The initial few cases can be computed directly for which the statement holds.
Assume the statement is true for $\omega_{0,n}$ with $n\leq |I|$.
For $\omega_{0,I+1}$,
we divide the terms in the recursive construction  \eqref{eqnWgI+1recursion}
of $\omega_{0,I+1}$ into two cases: those with $|J|,|K|>1$, and those with one of them equal to $1$.
For the first case, from the recursion, the $v$-dependent terms in the term
\begin{equation*}
\omega_{0,J+1}(v,u_{J}) \omega_{0,K+1}(v^{*}, u_{K})
\end{equation*}
with $|J|,|K|>1$ (and hence $|I|>3$),
are differential polynomials in $S(T+\delta_{r})$ where $\delta_{r}\in R^{\circ}\cup \{0\}$, with coefficients lying in $\widehat{\mathcal{K}}$.
Pick any term among all possible ramification points and all partitions
in the sum for the recursion.
From \eqref{eqnLaurentexpansionofwp} and \eqref{eqnexpansionforSchifferkernel} we see that $\omega_{0,J+1}(v,u_{J}) \omega_{0,K+1}(v^{*}, u_{K})$ is an element in
\begin{equation}
 \widehat{\mathcal{K}}[E_{2k+2}, k\geq 1\,, S^{(m\geq 0)} (\delta), \delta\neq 0](\!(  T )\!)
\otimes \mathbb{C}[S^{(m\geq 0)} (u_{i}-u_{r}) ,i\in I=J\cup K]\,.
\end{equation}
We introduce the notation $[-]_{n}$
for the degree $n$ Laurent coefficient at the corresponding point.
We also denote the $m$th derivative by the superscript $(m)$.
Then the ring above is
\begin{equation}
  \widehat{\mathcal{K}}\left[[S]_{m\in \mathbb{Z}} (\delta), \delta\in R^{\circ}\cup \{0\}\right](\!(  T )\!)  \otimes \mathbb{C}[S^{(m\geq 0)} (u_{i}-u_{r}) ,i\in I=J\cup K ] \,.
\end{equation}
For the second case where one of the cardinalities $|J|, |K|$, say $|J|$, is $1$, the ring is changed to
\begin{equation}
  \widehat{\mathcal{K}}\left[[S]_{m\in \mathbb{Z}} (\delta), \delta\in R^{\circ}\cup \{0\},  S^{(m\geq 0)} (u_{r}-u_{J}) \right](\!(  T )\!) \otimes \mathbb{C}[S^{(m\geq 0)} (u_{k}-u_{r}) ,k\in K ] \,.
\end{equation}

We also have from \eqref{eqnnumeratorofrecursionkernel} that
\begin{equation}
d^{-1}S \in \mathbb{C}\left[ S^{(m\geq 0)}(u_{r}-u_{0})\right]  [\![  T ]\!] \,.
\end{equation}
Applying chain rule to \eqref{eqndifferenceoflambda}, we obtain
\begin{eqnarray}\label{eqnrefinedringstructure}
&&\Lambda =\sum_{m\geq 2}[\Lambda]_{m}T^{m}\nonumber \\
&&    \in \mathbb{C}\left[   {1\over x}\rvert_{u_{r}}, {1\over y+y^{*}}\rvert_{u_{r}}, x^{(m\geq 1)}\rvert_{u_{r}}, (y-y^{*})^{(m\geq 1)}\rvert_{u_{r}},
  (y+y^{*})^{(m\geq 0)}\rvert_{u_{r}}
\right] T^{2}  [\![  T ]\!] \,.
\end{eqnarray}
Assumption  \ref{lemwpuniformizationofoneparameterfamiliesgenusonecurves} (which is true for our examples) for uniformization shows that (recall the expression for
$y^{*}$ from \eqref{eqnshyperellipticinvolutionalgebraic}),
\begin{equation}
x,y,y^{*}=-y-2h(x)\in \mathcal{M}(\Gamma)\otimes \mathbb{C}(\wp(u+\epsilon),\wp'(u+\epsilon))\,.
\end{equation}
 Lemma
\ref{lemmodularityofvaluesatramificationpoints} shows that
$x|_{u_{r}},y|_{u_{r}}$ and hence
\begin{equation}\label{eqnmodularityofvaluesatramification}
\wp(u_{r}+\epsilon), \wp'(u_{r}+\epsilon)\in \mathcal{M}(\Gamma(2)\cap \Gamma)\,,
\end{equation}
as the map from $(x,y)$ to $(\wp(u+\epsilon), \wp'(u+\epsilon))$
is a bi-regular map with coefficients being elements in $\mathcal{M}(\Gamma)$ from uniformization.
From the algebraic relation \eqref{eqnalgebraicrelationwpwp'} between $\wp,\wp'$, we see that
\begin{equation}
\wp^{(m\geq 0)}(u_{r}+\epsilon)\in  \mathcal{M}(\Gamma(2)\cap \Gamma)\,.
\end{equation}
Combing the above results we obtain
$\Lambda  \in \mathcal{M}(\Gamma(2)\cap \Gamma)  T^{2} [\![  T ]\!] $ and hence
\begin{equation}\label{eqnLambdaexpansionismodular}
{1\over \Lambda}\in  \mathcal{M}(\Gamma(2)\cap \Gamma)\, T^{-2} [\![  T ]\!] \,.
\end{equation}
Due to the order of pole behavior in Part 1, all of the formal Laurent and power series above can be replaced by their finite
truncations depending on $g,n$.
Multiplying the expansions of the above ingredients and collecting the degree $-1$ coefficients,
we see that $\omega_{0,I+1}$ is a differential polynomial in $S(u_{i}-u_{r}), i\in I\cup \{0\}, r\in R^{\circ}$,
and the coefficients are elements in the ring
\begin{equation}\label{eqnrefinedcoefficientring}
  \widehat{\mathcal{K}}\left[[S]_{m\in \mathbb{Z}} (\delta), \delta\in R^{\circ}\cup \{0\}\right]
  \otimes
\mathcal{M}(\Gamma(2)\cap \Gamma)
 \,.
\end{equation}
The results in Section \ref{secmodularityattorsion} tells that
$[S]_{m\in \mathbb{Z}-\{0\}} (\delta),\delta \in  R^{\circ}\cup \{0\}$ are weight-two holomorphic modular forms for $\Gamma(2)$ with trivial multiplier systems,
while we have
\begin{equation}
\{[S]_{0} (\delta) ,\delta \in R^{\circ}\cup \{0\}\}=\{  \widehat{e}_{1} ,  \widehat{e}_{2} , \widehat{e}_{3} , \widehat{\eta}_{1}  \}\,.
\end{equation}
Since $\mathcal{M}(\Gamma)\otimes M(\Gamma(2))\subseteq \mathcal{M}(\Gamma(2)\cap \Gamma)$,
the statement on the ring then follows.

The higher genus differentials are constructed from the genus zero ones. Since all ingredients are differential polynomials with coefficients in the ring $\widehat{\mathcal{K}}$, the conclusion follows automatically.

Observe that taking the $u$-derivative of an almost-meromorphic Jacobi form of index $0$ increases the weight by one.
As long as its Laurent coefficients are concerned, the recursion kernel $K$
can be regarded as an almost-meromorphic Jacobi form of weight $0$.
By tracing the degrees in the recursion formula \eqref{eqnWgI+1recursion},
and the weight $2$ of $\widehat{\omega}_{0,2}$ computed before,
we then immediately see the total weight of $\widehat \omega_{g,n}$
as an almost-meromorphic Jacobi form is $n$.

\item
This follows from the proof of Part 2 and the definition of $\hat{F}_{g}$ in \eqref{eqnFgrecursion}.

\end{enumerate}
\end{proof}

According to the proof of Remodeling Conjecture \cite{BKMP2009, FLZ16}, the GW potentials $d_{X_{1}}\cdots d_{X_n}F_{g,n}
$ and $F_{g}$ for the
 toric CY 3-fold $\cX$ coincide with
the differentials $\omega_{g,n}$ and $F_{g}$ produced by topological recursion for the mirror curve, using the Bergmann kernel $B$.
Observe that the non-holomorphic dependences in $\tau$ of $\widehat \omega_{g,n}, \hat{F}_{g}, 2g-2+n>0$ are polynomial in
$1/ \mathrm{Im} \tau$.
Taking the holomorphic limit, we arrive at the following easy consequence of Theorem \ref{thmhighergenusWgn}.
\begin{thm}\label{thmholhighergenusWgn}
	Consider the local toric Calabi-Yau 3-folds $\cX=K_{S}, S=\mathbb{P}^{2}, W\mathbb{P}[1,1,2], \mathbb{P}^{1}\times
	\mathbb{P}^{1}, \mathbb{F}_{1}$. Consider non-trivial one-parameter subfamilies of the mirror curves with  hyperelliptic structure determined by the corresponding brane.
	The following statements hold.

		The GW potentials $d_{X_{1}}\cdots d_{X_n}F_{g,n}=\omega_{g,n},2g-2+n>0, n>0$, as the holomorphic
		limits of the differentials $\widehat{\omega}_{g,n}(u_1,...u_n)$ which are almost-meromorphic Jacobi forms, are quasi-meromorphic Jacobi forms.
		The structure as quasi-meromorphic Jacobi forms is as exhibited in Theorem
		\ref{thmhighergenusWgn}, with the Schiffer kernel $S$ replaced by the Bergmann kernel $B$.

		The GW potentials $F_{g}=\omega_{g,0}, g\geq 2$,
		as the holomorphic limits of the differentials $\widehat{\omega}_{g,0}$ which are almost-meromorphic modular forms, are meromorphic quasi-modular forms lying
		in the ring
		\begin{equation}\label{eqnhollimitcoefficientring}
	\mathcal{K}:=\mathcal{M}(\Gamma(2)\cap \Gamma) \otimes \mathbb{C}[e_{1},e_{2},e_{3} ] [\eta_{1}] =\mathcal M(\Gamma(2)\cap\Gamma)\otimes \bC[ E_2]\,.
	\end{equation}
			\end{thm}

Recall that in all of our cases, the open GW point $\mathfrak s_0$ in \eqref{eqnopenGWpoint} given by $[x,y,1]=[0,-1,1]$
exists on the mirror curve $C$ independent of the generic complex parameters $(q_{1},\cdots)$.
The expansion of $F_{g,n}$ in terms of $X$
enumerates open GW invariants $\{n_{g,d,\mu}\}_{d,\mu}$
\begin{equation}\label{eqnopenGWgeneratingseries}
F_{g,n}=\sum_{ \mu\geq 1}X^{\mu} \sum_{d\geq 0}n_{g,d,\mu}Q^{d} \,,
\end{equation}
where $\mu=(\mu_{1},\cdots \mu_{n}),  X^{h}:=X_{1}^{\mu_{1}}\cdots X_{n}^{\mu_{n}}$.
See \eqref{eqn:Fgn} for the more detailed expression of this.
In our examples, after restriction to an one-parameter subfamily,
 we have \cite{Aganagic:2002, fang-liu-tseng} (for the $K_{\mathbb{P}^{2}}$ case there is no $q_{2}^{c_{2}}$ term)
\begin{equation}\label{eqnopenmirrormapofoneparameterfamily}
X_{k}=x_{k}\cdot c_{3}Q^{c}q_{1}^{c_{1}}q_{2}^{c_{2}}\,,
\end{equation}
for some $c, c_{1},c_{2}\in \mathbb{Q},c_{3}\in \mathbb{C}$.
Rewrite the generating series \eqref{eqnopenGWgeneratingseries} as
\begin{equation}\label{eqnopenGWgeneratingseriesshifted}
F_{g,n}=\sum_{ \mu\geq 1} (Q^{-c}X)^{\mu}\, \sum_{d\geq 0}n_{g,d,\mu}Q^{d+c\sum_{k}\mu_{k}} \,.
\end{equation}
The ring structure \eqref{eqnhollimitcoefficientring} in Theorem \ref{thmholhighergenusWgn} above
exhibits nice structure of the Taylor coefficients in this expansion.

\begin{cor}\label{corGWexpansionofWgn}
	With the same assumptions as Theorem \ref{thmholhighergenusWgn} above.
	The degree-$\mu$ Taylor coefficients $\sum_{d\geq 0}n_{g,d,\mu}Q^{d+c\sum_{k}\mu_{k}} $ in the expansion \eqref{eqnopenGWgeneratingseriesshifted} of $F_{g,n}$
	are meromorphic quasi-modular forms in the ring
	 $\mathcal{K}$ in \eqref{eqnhollimitcoefficientring}.
				\end{cor}
\begin{proof}
Part 1 of Theorem \ref{thmhighergenusWgn} tells that generically
the differential
$\omega_{g,n}$ does not have singularity at the open GW point \eqref{eqnopenGWpoint} which avoids the ramification points.
Hence developing Taylor expansion makes sense
and we have, recall that $\mu_{k}\geq 1$,
\begin{eqnarray}
&&\sum_{d}n_{g,d,\mu}Q^{d+c\sum_{k}\mu_{k}}\nonumber\\
& =& {1\over \prod_{k=1}^{n}\mu_{k}!}\prod_{k=1}^{n}{\partial^{\mu_{k}} \over \partial (Q^{-c}X_{k})^{\mu_{k}}}\vert_{X=0} F_{g,n}\nonumber \\
&=& {1\over \prod_{k=1}^{n}\mu_{k}!}\prod_{k=1}^{n}{\partial^{\mu_{k}-1} \over \partial (Q^{-c}X_{k})^{\mu_{k}-1}}\vert_{X=0}
{ \omega_{g,n}\over du_{1}\boxtimes \cdots \boxtimes du_{n}}
{1\over \prod_{k=1}^{n}   \partial_{u_{k}} (Q^{-c}X_{k})}\,.
\end{eqnarray}
Theorem \ref{thmholhighergenusWgn}
shows that
\begin{equation}
{ \omega_{g,n}\over du_{1}\boxtimes \cdots \boxtimes du_{n}}\in \mathcal{K}[\wp^{(m\geq 0)} (u_{k}-u_{r}), k\in \{1,2,3\dots n\}, r\in R^{\circ}]\,.
\end{equation}
By using the algebraic relation \eqref{eqnalgebraicrelationwpwp'} between $\wp$ and $\wp'$, the above ring can be reduced to
\begin{equation}
 \mathcal{K}[\wp (u_{k}-u_{r}), \wp'(u_{k}-u_{r}), k\in \{1,2,3\dots n\}, r\in R^{\circ}]\,.
\end{equation}
The chain rule says
\begin{equation}
\partial_{u} (Q^{-c}X)= \partial_{x} (Q^{-c}X) \cdot  \partial_{u}x\,,
\quad
{\partial\over \partial  (Q^{-c}X)}={1\over \partial_{x} (Q^{-c}X)} {1\over \partial_{u}x} {\partial\over \partial u}\,.
\end{equation}
From \eqref{eqnopenmirrormapofoneparameterfamily}
we obtain
\begin{equation}
\partial_{x_{k}} (Q^{-c}X_{k})=c_{3}q_{1}^{c_{1}}q_{2}^{c_{2}}\,.
\end{equation}
According to the discussion in Section \ref{sec:one-parameter}
and Section
\ref{secallfourexamples}, after the restriction to an one-parameter subfamily, both $q_{1},q_{2}$
become modular functions for a certain modular group $\Gamma$ depending on the subfamily.
Hence so is $\partial_{x_{k}} (Q^{-c}X_{k})$, where the subtlety of taking roots of modular functions can be addressed similarly as in
Section
\ref{secallfourexamples}.
The same argument in establishing
\eqref{eqnLambdaexpansionismodular}
in the proof of Part 2 of Theorem  \ref{thmhighergenusWgn} shows that
\begin{equation}\label{eqnmodularityofvaluesatopenGWofxy}
\wp(u+\epsilon)\rvert_{u=u_{\mathfrak{s}_{0}}}\,, \quad
\wp'(u+\epsilon)\rvert_{u=u_{\mathfrak{s}_{0}}}\,,\quad
 \partial_{u}^{m\geq 0}x\rvert_{u=u_{\mathfrak{s}_{0}}}\in \mathcal{M}( \Gamma)\,.
\end{equation}
This implies that the values at the open GW point
of terms arising from differentials of the term $\partial_{u_{k}}x_{k}$
also lie in $\mathcal{M}( \Gamma)$.

To prove the desired statement, it remains to show
\begin{equation}
\wp(u_{k}-u_{r})\rvert_{u=u_{\mathfrak{s}_{0}}}\,, \quad \wp'(u_{k}-u_{r})\rvert_{u=u_{\mathfrak{s}_{0}}}\in \mathcal{M}(\Gamma(2)\cap \Gamma)\,.
\end{equation}
This is automatically
true for those cases in which
$u_{\mathfrak{s}_{0}}$ is identified with a torsion point
according to the discussion in Section \ref{secmodularityattorsion}.
In general, we use \eqref{eqnmodularityofvaluesatramification},  \eqref{eqnmodularityofvaluesatopenGWofxy} and the addition formula for $\wp$ which tells that
\begin{equation}\label{eqnadditionformulaforwp}
\wp(u_{\mathfrak{s}_{0}} -u_{r})
={1\over 4}
\left({
\wp'(u_{\mathfrak{s}_{0}}+\epsilon)+\wp'(u_{r}+\epsilon)
\over \wp(u_{\mathfrak{s}_{0}}+\epsilon)-\wp(u_{r}+\epsilon)}
\right)^2
-\wp(u_{\mathfrak{s}_{0}}+\epsilon)
-\wp(u_{r}+\epsilon)\,.
\end{equation}
\end{proof}

\begin{rem}
	For each $g, n$, $ \omega_{g,n}$ is an $n$-variable differential polynomials in $\wp$.
	By carefully keeping track of the degrees in the generators including the derivatives of the Weierstrass-$\wp$
	functions and the meromorphic quasi-modular forms basing on the structure of the coefficient ring in \eqref{eqnrefinedringstructure}, we can see that
	for each fixed $n$, there are only finitely many possible terms (see e.g. Example \ref{exKP2continued}) with numbers being coefficients.
	Again using the algebraic relation \eqref{eqnalgebraicrelationwpwp'} between $\wp$ and $\wp'$, we can futher reduce the number of generators since differential polynomials in $\wp$
	are polynomials in $\wp,\wp'$.
	This structure tells that determining $ \omega_{g,n}$ can be reduced to a finite computation.
	In particular, knowing the first few terms (depending on $g, n$) in the expansion of $ \omega_{g,n}$,
	which can in principle be computed from the
	A-model of the mirror symmetry side, would then be enough to fix $ \omega_{g,n}$ completely.
\end{rem}

\subsection{Holomorphic anomaly equations}
\label{secholomorphicanomalyequations}

In \cite{Bershadsky:1993ta, Bershadsky:1993cx}, it is argued from physics that the closed string free energies $\widehat{\mathcal{F}}_{g}$ satisfy a system of recursive equations called holomorphic anomaly equations (HAE). We recall the set-up of HAE adapted to our local cases from \cite{Konishi:2010local}.\\

Recall that we have a family of mirror curves $\chi: \cC\to \cU_\cC$, which is one dimensional when $S=\bP^2$, and is two dimensional for the other three cases.

When the dimension of $\cU_\cC$ is $2$, as a special case of \cite[Definition 6.1]{Konishi:2010local}, there is a rank $1$ subbundle $T^0 \cU_\cC \subset T \cU_\cC$. It is characterized by
\[
v\in T^0\cU_\cC \iff
\partial_v \left ( \int_{\tilde A} \lambda\right)=0,\ \forall \tilde A\in K^\circ(C^\circ;\bZ).
\]
Here $K^\circ(C^\circ;\bZ)=\ker(K_1(C^\circ;\bZ)\to H_1(C;\bZ))$. As shown in \cite{Konishi:2010local}, there are coordinates system $t_0, t_1$ on $\cU_\cC$ such that $\frac{\partial}{\partial t_0}\in T^0 \cU_\cC$. We choose primitive cycles $\tilde A'_0,\tilde A'_1\in K_1(C^\circ;\bZ)$ and local coordinates (near $C$ in the family $\cC$)
\[
t_0=\frac{1}{2\pi\sqrt{-1}}\int_{\tilde A'_0} \lambda,\ t_1=\frac{1}{2\pi\sqrt{-1}}\int_{\tilde A'_1}\lambda\,,
\]
such that $\tilde A'_0, \tilde A'_1$ are linearly equivalent to $\tilde A_0,\tilde A_1$ in $K_1(C^\circ;\bC)$, and that $\tilde A'_1 \in K^\circ(C^\circ;\bZ)$. The cycles $\tilde A_0$ and $\tilde A_1$ are defined in Section \ref{sec:remodeling} to yield mirror maps. By the special geometry property, for any $\tilde A\in K^\circ (C^\circ;\bZ)$,
\[
\frac{\partial}{\partial t_0} \int_{\tilde A} \lambda=\int_{\tilde A}\omega_0=0\,,
\]
where $\int_{\tilde A_0'} \omega_0=1$ and $\omega_0\in \Omega^1(C)$. So for all of our four examples we have local coordinates $\{t_a\}_{a=0}^{\fp-1}$ on $\cU_\cC$ where $\fp=h^2(\cX)=1$ or $2$, and $\frac{\partial}{\partial t_0}\in T^0 \cU_\cC$.

Let $A\in H_1(C;\bZ)$ be the image of $\tilde A_0'\in K_1(C^\circ;\bZ)$, and $A,B$ be a symplectic basis of $H_1(C;\bZ)$ as before. We define (for $k=0,\dots, \fp-1$)
\begin{align*}
G_{0\bar 0}&=-{\sqrt{-1}}\int_C \omega_0 \wedge \bar \omega_0=-\sqrt{-1}(\tau-\bar \tau),\\
C_{00k}&=\sqrt{-1}\int_C \nabla_{ k}\omega_0 \wedge \omega_0\,.
\end{align*}
The quantity $G_{0\bar 0}$ defines a Hermitian metric (analogue of Weil-Petersson metric in the compact case) on $T^0\cU_\cC$: $G(\frac{\partial}{\partial t_0},\frac{\partial}{\partial t_0})=G_{0\bar 0}$, while $\nabla$ is the Gauss-Manin connection and $C_{00k}$ is called the Yukawa coupling. Let $C^{00}_{\bar k}=\overline{C_{00k}} (G_{0\bar 0})^{-2}.$ For $g>1$, the holomorphic anomaly equations are (see (7.5) of \cite{Konishi:2010local})
\begin{align}
\bar{\partial}_{\bar{k}}\hat{F}_{g}
&={1\over 2}
C^{00}_{\bar{k}}
\left(
D_{0}D_{0}\hat{{F}}_{g-1}
+\sum_{g_1=1}^{g-1}
D_{0}\hat{{F}}_{g_{1}} \cdot D_{0}\hat{{F}}_{g-g_1}
\right)\nonumber \\
&={1\over 2}
C^{00}_{\bar{k}}
\left(
\partial_{t_0}\partial_{t_0}\hat{{F}}_{g-1}
+\kappa\frac{\partial \tau}{\partial t_0}\partial_{t_0}\hat F_{g-1}+\sum_{g_1=1}^{g-1}
\partial_{t_0}\hat{{F}}_{g_{1}} \cdot \partial_{t_0}\hat{{F}}_{g-g_1}\right)\,.
\label{eqnHAE}
\end{align}
In this equation, the second $D_0$ on $\hat F_{g-1}$ (acting on $\hat F_{g-1}$ directly) and both $D_0$ on $\hat F_{g_1},\hat F_{g_2}$ are just $\frac{\partial}{\partial t_0}$: we just regard them as a connection on the trivial line bundle $\cL$ of which $\hat F_g$ is a smooth section along $\{t_1=\mathrm{constant}\}$
\[
D: \Gamma(\cL) \to \Gamma(\cL \otimes (T^0\cU_\cC)^*)\,.
\]
The operator $D_0$ acting on $D_0\hat F_{g-1}$ is (where $\kappa=-1/(\tau-\bar\tau)$ as defined in \eqref{eqn:modified-cycles})
\[
D_0=\frac{\partial}{\partial t_0}+\kappa \frac{\partial \tau}{\partial t_0}\,,
\]
which is the Chern connection from the Weil-Petersson metric on $T^0\cC$. We don't give a B-model definition for $\hat F_1$ in this paper, but note that $D_0 \hat F_1=\int_B \hat \omega_{1,1}$.\\

It is shown in \cite{Eynard:2007holomorphic} that $\hat F_g, g\geq 2$  produced by the topological recursion from any spectral curve (and in particular for our mirror curves in our cases) satisfy an equation like \eqref{eqnHAE}. Moreover $\widehat \omega_{g,n}$ produced from the topological recursion also satisfies a similar set of equations for $2g-2+n>0$
\begin{align}
&\bar{\partial}_{\bar{k}}\widehat{\omega}_{g,n}(p_1,\dots, p_n) \nonumber\\
=&{1\over 2} C^{00}_{\bar k}
\left(
D_{0}D_{0}\omega_{g-1, n}(p_1,\dots,p_n)
+\sum_{ \substack{g_{1}+g_{2}=g,\\
 J\sqcup K=\{1,\dots,n\}, \\
 (g_1,|J|)\neq (0,0),(g,n)}}
D_{0}\hat \omega_{g_{1}, |J|}(p_J) \cdot D_0\hat \omega_{g_{2}, |K|}(p_K)
\right)\,.\label{eqnHAEfortopologoicalrecursion}
\end{align}
In the equation we regard $\hat \omega_{g,0}=\hat F_g$, $\hat \omega_{0,1}=\lambda$ and $D_0 \hat \omega_{0,1}=\omega_0$. Equation \eqref{eqnHAE} is a special case of \eqref{eqnHAEfortopologoicalrecursion}.

\begin{rem}
It follows from \cite{Eynard:2007holomorphic} that the Yukawa coupling has an A-model description under mirror symmetry
\begin{equation}\label{eqnantiholYukawa}
C_{00k}= -\frac{1}{(2\pi)^2}\overline{\left ({\partial^3 F_0\over \partial  t_0^2\partial t_k}\right)}.
\end{equation}
The quantity $F_0$ is the A-model genus zero GW potential including the classical limit term of the intersection theory for equivariant cohomology.
\end{rem}

Now we translate the above differential equations \eqref{eqnHAE} and  \eqref{eqnHAEfortopologoicalrecursion} for the non-holomorphic (in $t_0$) differentials $\widehat{\omega}_{g,n}$, which are defined by using the Schiffer kernel $S$, into equations for the corresponding holomorphic differentials $\omega_{g,n}$ defined using the Bergman kernel $B$.

In proving the modularity results in the previous section we have restricted ourselves to non-trivial one-parameter subfamilies $\chi_{\mathrm{res}}: \cC_{\mathrm{res}}\rightarrow \mathcal{U}_{\mathrm{res}}$ for the three two-parameter family cases.
Here we restrict to a subfamily $\{t_1=\mathrm{const}\}$ which is more restrictive than our theorem for modularity (Theorem \ref{thmhighergenusWgn}). From Theorem \ref{thmhighergenusWgn}, we know that the
$\widehat{\omega}_{g,n}$'s
are polynomials
of almost-meromorphic Jacobi forms and almost-meromorphic modular forms, with the
only nontrivial non-holomorphic dependence in
$t_0$ entering through the Schiffer kernel $S$
and the non-holomorphic (in $\tau$) generators $\widehat{e}_{a}, a=0,1,2,3$ in \eqref{eqnnonholomorphicperiodsof2ndkind}.
Therefore, by the chain rule,
\begin{equation}
  \frac{\partial}{\partial \bar t_0}=
\sum_{a=0}^{3}{\partial  \widehat{e}_{a} \over \partial \bar{t}_0}  {\partial \over \partial \widehat{e}_{a}}
+
\sum_{k,r}
{\partial  S_{kr}\over \partial \bar{t}_0}  {\partial \over \partial S_{kr}}\,,
\end{equation}
where $S_{kr}=S(u_k-u_r)$ stands for the Schiffer kernel with argument $u_{k},u_{r}, k=1,2,\cdots n, r\in R$.
\begin{rem}
Here
 in order to apply the chain rule and hence translate the derivative
$\partial_{\bar{t}_{0}}$ into the derivative with respect to the generators,
we do not need the algebraic independence among the generators
$\hat{e}_{0}=\hat{\eta}_1, \hat{e}_{a}=e_{a}
+\hat{\eta}_1, a=1,2,3$ and $S_{kr},k=1,2\cdots n, r\in R$.

In fact, fixing $k$ and varying $r$, by using the addition formula
for $\wp$ as in \eqref{eqnadditionformulaforwp}, one can show that
the transcendental degree of of the field generated by
$e_{a}, a=1,2,3, \wp(u_k-u_r), r\in R$ over the field generated by
$e_{a}, a=1,2,3$
is $1$.
Moreover, the rational functions $\wp(u_k-u_r), \wp(u_k'-u_r), k\neq k'$ are algebraically independent over the field generated by
$e_{a}, a=1,2,3$, while $\hat{e}_{0}$ is algebraically independent of
any set of meromorphic quantities.
\end{rem}

From the explicit formulae for the generators $\widehat e_{a}$ in \eqref{eqnnonholomorphicperiodsof2ndkind} and for the Schiffer kernel $S$ in  \eqref{eqngenusoneSchifferkernel}, this can be simplified into
\begin{equation}
{\partial  \hat{\eta}_{1}\over \partial \bar{t}_0} \sum_{a=0}^{3} {\partial \over \partial \widehat{e}_{a}}
+
{\partial  \hat{\eta}_{1}\over \partial \bar{t}_0} \sum_{k,r}  {\partial \over \partial   S_{kr}}\,.
\end{equation}
Hence \eqref{eqnHAEfortopologoicalrecursion} becomes (for $I=\{1,\dots,n\}$ and $p_I=(p_1,\dots,p_n)$)
\begin{align}
&\left(\sum_{a=0}^{3} {\partial \over \partial \widehat{e}_{a}}
+
\sum_{k,r}  {\partial \over \partial   S_{kr}}\right)
\widehat{\omega}_{g,n}(p_I) \nonumber\\
&=
{1\over 2}{C^{00}_{\bar 0} \over { \partial_{\bar{t}_0} }  \hat{\eta}_{1} }\cdot
\left(
D_{0}D_{0}\hat\omega_{g-1, n}(p_I)
+\sum_{ \substack{g_{1}+g_{2}=g,\\I=J\sqcup K, (g_1,J)\neq(0,\emptyset),(g,I) }}
D_{0}\hat\omega_{g_{1},  |J|}(p_J) \cdot D_{0}\hat\omega_{g_{2}, |K| }(p_K)
\right)\,.
\end{align}
In the case of compact Calabi-Yaus, the term
$C^{00}_{\bar{0}}$ is usually rewritten  \cite{Bershadsky:1993cx} with the help of results computed from the Weil-Petersson metric on the moduli space. Define the propagator
$S^{00}$ to be
\begin{equation}\label{eqnclosedstringpropagator}
\bar{\partial}_{\bar{0}}S^{00}=C^{00}_{\bar{0}}\,.
\end{equation}
The computations for $S^{00}$ in \cite{Alim:2013eja}
for the local Calabi-Yau cases yield explicit results for them in terms of almost-holomorphic modular forms (actually we can take any solution to \eqref{eqnclosedstringpropagator} whose non-holomorphic dependence has no ambiguity). The structure theorem for almost-holomorphic modular forms
\cite{Kaneko:1995} tells that their nontrivial anti-holomorphic dependences are in polynomials in $Y:=-\pi /\mathrm{Im }\tau$. For the current cases the quantities $S^{00}$ are in fact linear in $Y$. This then leads to
\begin{equation}
{\partial_{\bar{t}_0}S^{00}\over \partial_{\bar{t}_0}\hat{\eta}_{1}}
={\partial_{\bar{\tau}} S^{00}\over \partial_{\bar{\tau}}\hat{\eta}_{1}}
={\partial_{Y} S^{00} \over \partial_{Y}\hat{\eta}_{1}}\,.
\end{equation}

The BCOV type holomorphic anomaly equation \eqref{eqnHAEfortopologoicalrecursion} for $\widehat{\omega}_{g,n}$
is finally translated into the Yamaguchi-Yau type \cite{Yamaguchi:2004bt}  functional equation
\begin{align}
&\left(\sum_{a=0}^{3} {\partial \over \partial \widehat{e}_{a}}
+
\sum_{k,r}  {\partial \over \partial   S_{kr}}\right)
\widehat{\omega}_{g,n}(p_I)  \nonumber\\
=&{1\over 2}
{\partial_{Y} S^{00} \over \partial_{Y}\hat{\eta}_{1}}
\left(
D_{0}D_{0}\hat\omega_{g-1, n}(p_I)
+\sum_{\substack{ g_{1}+g_{2}=g,\\
I=J\sqcup K, (g_1,J)\neq(0,\emptyset),(g,I) }}
D_{0}\hat\omega_{g_{1},  |J|}(p_J) \cdot D_{0}\hat\omega_{g_{2}, |K| }(p_K)
\right)\,.
\end{align}

Due to the structure for
$\widehat{\omega}_{g,n}$ in Theorem \ref{thmhighergenusWgn},
this identity is an identity for polynomials in $Y$ (with coefficients being holomorphic quantities).
Therefore, we can take the degree zero term in $Y$ (called the holomorphic limit).
Observe that the holomorphic limit of the holomorphic derivatives of $Y$ vanish in the holomorphic limit.
This then yields a functional equation for the differentials $\omega_{g,n}$ produced by using the Bergmann kernel $B$
(in what follows $B_{kr}=B(u_{k}-u_{r})$)
\begin{align}
&\left( {\partial \over \partial \eta_{1}}
+
\sum_{k,r}  {\partial \over \partial   B_{kr}}\right)
\omega_{g,n}(p_I)  \nonumber\\
=&{1\over 2}
{\partial_{Y} S^{00} \over \partial_{Y}\hat{\eta}_{1}}
\left(
\partial_{t_0}\partial_{t_0}\hat\omega_{g-1, n}(p_I)
+\sum_{ \substack{g_{1}+g_{2}=g,\\
I=J\sqcup K, (g_1,J)\neq(0,\emptyset),(g,I) }}
\partial_{t_0}\hat\omega_{g_{1},  |J|}(p_J) \cdot \partial_{t_0}\hat\omega_{g_{2}, |K| }(p_K)
\right)\,.
\label{eqn:YY-open}
\end{align}

Note that the other generators discussed in Theorem \ref{thmhighergenusWgn}
are considered to be independent of $B$. The reason is that they are so before the holomorphic limit: $S$ includes the transcendental quantity $Y$ while the others do not. Plainly, that $B$ is not modular permits us to distinguish it from the rest of the generators which are all modular. This is what makes $B$ algebraically independent of the rest. 

Combing the proof of Remodeling Conjecture for toric CY's in \cite{BKMP2009, FLZ16}, it follows then that the GW potentials satisfy the above Yamaguchi-Yau type functional equations. We summarize the results in the following theorem.
\begin{thm}\label{thmhae}
	Consider the local toric Calabi-Yau 3-folds $\cX=K_{S}, S=\mathbb{P}^{2}, W\mathbb{P}[1,1,2], \mathbb{P}^{1}\times
	\mathbb{P}^{1}, \mathbb{F}_{1}$. Consider a non-trivial one-parameter subfamily $\chi_{\mathrm{res}}: \cC_{\mathrm{res}}\rightarrow \mathcal{U}_{\mathrm{res}}$ such that $t_1$ is a constant. The GW potentials $\omega_{g,n},2g-2+n>0$ satisfy the Yamaguchi-Yau-type holomorphic anomaly equations \eqref{eqn:YY-open}.
  The quantity $S^{00}$ is defined to be a solution to \eqref{eqnclosedstringpropagator}.
  As a special case, the closed GW potentials $F_{g}$ $(g>1)$  satisfy \begin{align}
  &{\partial \over \partial \eta_{1}}
  F_g
  ={1\over 2}
  {\partial_{Y} S^{00} \over \partial_{Y}\hat{\eta}_{1}}
  \left(
  \partial_{t_0}\partial_{t_0} F_{g-1}
  +\sum_{ \substack{g_{1}+g_{2}=g,\\
  g_1\neq 0,g }}
  \partial_{t_0}F_{g_{1}} \cdot \partial_{t_0}F_{g_{2}}
  \right)\,.
  \label{eqn:YY-closed}
  \end{align}
\end{thm}

\begin{ex}[$K_{\mathbb{P}^{2}}$ continued]
\label{exKP2continued}

The natural parameters in the generating series of open GW invariants
are the closed modulus $T$ and the open modulus $X$.

The closed modulus $T$ is the K\"ahler normal coordinate with respect to
Weil-Petersson metric on the moduli space of K\"ahler structures of the CY 3-fold $K_{\mathbb{P}^2}$,
near the large volume limit.
In the B-model this is the flat coordinate, defined as a period integral in Section \ref{sec:one-parameter}.
Explicitly it is,  see \cite{Chiang:1999tz, Aganagic:2000gs, Aganagic:2002, Batyrev:1993variations, Stienstra:1997resonant, Hosono:2004jp, Konishi:2010local},
\begin{equation}\label{eqnclosedmodulus}
T=\log (-1)+\log {-q_{1} \over 27}+\sum_{k\geq 1} {(3k)!\over (k!)^{3}} {1\over k} ({-q_{1}\over 3^{3}})^{k}\,,
\quad
q_{1}=-3^{3}{\eta(3\tau)^{9}\over \Theta^{3}_{A_{2}}(2\tau)\eta(\tau)^{3}}\,.
\end{equation}
Its derivative in the variable $\log (-q_{1})$ is related to the $\theta$-function of the $A_{2}$-lattice
and is a modular form for $\Gamma_{0}(3)$.
The quantity $Q=e^{T}$ is related to modular variable $e^{2\pi i \tau}$
of the mirror curve
by an infinite product \cite{Mohri:2001zz, Stienstra:2005wy, Zhou:2014thesis}.

The open modulus $X$ is described by an integral along a carefully chosen chain in $C$.
According to \cite{Aganagic:2000gs, Aganagic:2002}, one has
\begin{equation}\label{eqnopenmodulus}
X
=\exp\left(\log x+\log (-1)+{1\over 3}\sum_{k\geq 1} {(3k)!\over (k!)^{3}} {1\over k} ({-q_{1}\over 3^{3}})^{k}\right)\,.
\end{equation}
With this choice, the open modulus and the affine coordinate $u$ on the Jacobian is related by using the uniformization in Section \ref{secexKP2}
\begin{eqnarray}\label{eqntomodulus}
X&=&
 \left((-4)^{1\over 3}\kappa^{2} \wp +{3\over 4}\phi^{2} \right)\cdot\exp\left( {1\over 3} \left(T +2 \log (-1)-\log {-q_{1}\over 27}\right)\right)\,.
\end{eqnarray}
Thanks to the identification in Section
\ref{secexKP2}
that the open GW point \eqref{eqnopenGWpoint} is a $3$-torsion point and the results in
Section \ref{secmodularityattorsion}, the coefficients in the expansion in $X$ of the GW potentials $\{\omega_{g,n}\}_{g,n}$
are meromorphic quasi-modular forms in $\tau$.\\

The ring \eqref{eqnhollimitcoefficientring} in Theorem \ref{thmholhighergenusWgn} is a subring of the following
	\begin{equation}
	\mathbb{C}[e_{1},e_{2},e_{3},\eta_{1}]\left[{1\over 1-3\phi x(u_{r})}, {1\over x(u_{r})}, \kappa,\kappa^{-1}, \phi, {1\over \wp''(u_{r})}, \wp^{(m\geq 2)} (u_{r}), u_r\in  \{{1\over 2}, {\tau\over 2}, {1+\tau\over 2}\}\right].
	\end{equation}
		Regarded as a polynomial in $\eta_{1}$,
	the coefficient of any element in this ring is a meromorphic modular form of level
$\Gamma(2)\cap \Gamma_{0}(9)$ as shown in  Section \ref{secexKP2}.
	Using Theorem \ref{thmhighergenusWgn} and the algebraic relation \eqref{eqnalgebraicrelationwpwp'} between $\wp,\wp'$, we see that ${\omega}_{g,n}$ lies in a ring with only finitely many generators.\\

In the computation of genus one free energy,
using the uniformization in Section \ref{secexKP2} it is straightforward to compute
\begin{equation}
 {dy\over d(x-x(r))^{1\over 2}}|_{r}={\partial_{u}y\over (2^{-1}\partial_{u}^2 x )^{1\over 2}}|_{r}=\kappa^2({2\wp''(u_{r})\over (-4)^{1\over 3} })^{1\over 2}\,.
\end{equation}
Using the results in \eqref{eqnperiodsof2ndkind}, we obtain
\begin{equation}
\prod_{r}\wp''(u_{r})=-{1\over 2}\Delta=-{1\over 2}(2\pi)^{12}\eta^{24}\,.
\end{equation}
where $\Delta$ is the Dedekind $\Delta$-function
and $\eta$ is the $\eta$-function.
Hence we get, up to addition by constant,
\begin{equation}
-{1\over 12}\ln \prod_{r} {dy\over d(x-x(r))^{1\over 2}}|_{r}=-{1\over 12}\ln (\kappa^{6}\eta^{12}) \,.
\end{equation}
Combining the above formula for the Bergmann $\tau$-function, we therefore get
\begin{eqnarray}
\hat{F}_{1}&=&
-{1\over 2}\ln \tau_{B}-{1\over 12}\ln \prod_{r} {dy\over d(x-x(r))^{1\over 2}}|_{r}+{1\over 2}\ln \det Y \nonumber \\
&=&-{1\over 2}\log \left(\eta(\tau)\eta(3\tau) \sqrt{\mathrm{Im}\tau}\sqrt{\mathrm{Im}3\tau}\right)\,.
\end{eqnarray}
This agrees with the results in \cite{Aganagic:2006wq, Haghighat:2008gw, Alim:2013eja}
obtained by other means.\\

For the CY 3-fold $K_{\mathbb{P}^{2}}$,  the mirror curve family is an one-parameter family.
Under the flat coordinate $t_0=\frac{1}{3}T$, we have from \cite{Alim:2013eja} (see also \cite{Zhou:2014thesis}) that
\begin{equation}
S^{00}={1\over 2} { 3E_{2} (3\tau)+E_{2}(\tau)\over 4}+{1\over 2} {-3\over \pi \mathrm{Im} \tau}
={1\over 2} { 3E_{2} (3\tau)+E_{2}(\tau)\over 4}+{3\over 2 \pi^2} Y\,.
\end{equation}
Hence in Theorem \ref{thmhae} we have
\begin{equation}
{\partial_{\bar{t}_0}S^{00}\over \partial_{\bar{t}_0}\hat{\eta}_{1}}
={\partial_{Y} S^{00} \over \partial_{Y}\hat{\eta}_{1}}={3\over 2\pi^2}\,.
\end{equation}
\end{ex}

\begin{appendices}

\section{Some explicit formulae }
\label{secappendix}
Some explicit formulae for the disk potential, annulus potential, $\omega_{0,3}$, and $\omega_{1,1}$ for certain special one-parameter families of our four examples are collected in this appendix.
The general expressions are displayed below.
\begin{itemize}
\item Disk potential
\begin{equation}
\partial_x{W}=\log y\cdot {1\over x}\,.
\end{equation}
\item Annulus potential
\begin{equation}
\omega_{0,2}(u_{1},u_{2})=B(u_{1},u_{2})=(\wp(u_{1}-u_{2})+ {\eta}_{1})du_1\boxtimes du_2\,.
\end{equation}

\item Recursion kernel $K=d^{-1}S/\Lambda$,
\begin{eqnarray}
S(u_{1},u_{2})&=&(\wp(u_{1}-u_{2})+\widehat {\eta}_{1}) du_1\boxtimes du_2\,,\nonumber\\
 \Lambda&=&2 \sum_{k=0}^{\infty} {1\over 2k+1} ({y-y^{*}\over y+y^{*}})^{2k+1}\partial_{u}x {1\over x} du\,.
\end{eqnarray}
Here $d^{-1}S$
is as defined in \eqref{eqnnumeratorofrecursionkernel}, and the expression $y^{*}=-y-2h(x)$ in \eqref{eqnshyperellipticinvolutionalgebraic} is determined from the mirror curve equation as in \eqref{eqnhyperellipticformofmirrorcurve} and \eqref {eqnsimplechangeonybyh}.

\item $\omega_{0,3}$
\begin{equation}
\omega_{0,3}(u_{1},u_{2},u_{3})=
\sum_{r\in R^{\circ}}  \left(  2[{1\over \Lambda} ]_{-2}\cdot
\prod_{k=1}^{3}
(\wp(u_{k}-u_{r})+\eta_{1})
\right) du_1\boxtimes du_2\boxtimes du_{3}\,,
\end{equation}

\item $\omega_{1,1}$
\begin{equation}
\omega_{1,1}(u_1) =\sum_{r\in R^{\circ}}\left(
{1\over 24}[{1 \over\Lambda}]_{-2}\wp^{(2)}(u_1-u_r)+{{\eta_1}}[{1\over \Lambda }]_{-2}\wp(u_1-u_r)+{1\over 4}[{1\over \Lambda }]_{0}\wp(u_1-u_r)
\right)du_1\,.
\end{equation}
\end{itemize}
In the above we have used the notation $[-]_{n}$
to denote the degree $n$ Laurent coefficient at the corresponding point in consideration.
Direct computations show that
\begin{eqnarray}
[{1\over \Lambda}]_{-2}=
{1\over  [ \Lambda]_{2}}={1\over a_{0}}
\,,
\quad
[{1\over \Lambda }]_{0}=
-{a_{2}\over a_{0}^{2}}+{a_{1}^{2}\over a_{0}^{3}}\,,
\end{eqnarray}
where
\begin{eqnarray}
a_{0}&=&{2 [x']_{1} [y-y^{*}]_{1}\over    [x]_{0} [y+y^{*}]_{0} }\,,\\
a_{1}&=& {2 [x']_{1} [y-y^{*}]_{1}} \cdot - {   [x]_{0} [y+y^{*}]_{1}+ [x]_{1} [y+y^{*}]_{0}  \over    [x]_{0}^2 [y+y^{*}]^2_{0} }
\\&+&
2{  [x']_{1} [y-y^{*}]_{2}+ [x']_{2} [y-y^{*}]_{1} \over    [x]_{0} [y+y^{*}]_{0} }
\,,\\
a_{2}&=& {2 [x']_{1} [y-y^{*}]^{3}_{1}\over   3  [x]_{0} [(y+y^{*})^3]_{0} } \\
&
+& {2 [x']_{1} [y-y^{*}]_{3}+2 [x']_{2} [y-y^{*}]_{2}+2 [x']_{3} [y-y^{*}]_{1}\over [x]_{0} [(y+y^{*})]_{0}}\\
&-&
 { 2  (  [x']_{1} [y-y^{*}]_{2} + [x']_{2} [y-y^{*}]_{1})  ( [x]_{0} [y+y^{*}]_{1} +[x]_{1} [y+y^{*}]_{0})     \over     [x]_{0}^2 [(y+y^{*})]_{0}^2 } \\
 &+&
 {2  (  [x']_{1} [y-y^{*}]_{1}) ( [x]_{2} [y+y^{*}]_{0} +[x]_{1} [y+y^{*}]_{1}+[x]_{0} [y+y^{*}]_{2})     \over     [x]_{0}^2 [(y+y^{*})]_{0}^2 } \\
  &-&
 {2  (  [x']_{1} [y-y^{*}]_{1}) ( [x]_{1} [y+y^{*}]_{0} +[x]_{1} [y+y^{*}]_{0})^2     \over     [x]_{0}^3 [(y+y^{*})]_{0}^3 }
\,.
\end{eqnarray}

\subsection{$K_{\mathbb{P}^{2}}$}
The affine part of the mirror curve given in Example \ref{ex:1} is equivalent to
\begin{equation}
y^{2}+(x+1)y+q_{1}x^{3}=0\,,
\quad
q_{1}=(-3\phi)^{-3}\,.
\end{equation}
The set of finite ramification points is
$R^{\circ}=\{{1\over 2}, {\tau\over 2}, {1+\tau\over 2}\}$.
Uniformization gives
\begin{equation}
x=-3(-4)^{1\over 3}\kappa^{2} \phi \wp (u)-{9\over 4}\phi^{3}\,,
\quad
y=\kappa^{3}\wp '(u)-{1+x\over 2}\,.
\end{equation}
with
\begin{equation}
\phi(\tau)=\Theta_{A_{2}}(2\tau) {\eta(3\tau) \over \eta(\tau)^{3}  }\,,
\quad
\kappa=\zeta_{6}\,2^{-{4\over 3}} 3^{1\over 2} \pi^{-1}{\eta(3\tau)\over\eta(\tau)^{3}}\,.
\end{equation}

\subsection{$K_{\mathbb{F}_1}$}

The affine part of the mirror curve given in Example \ref{ex:4} is
\begin{eqnarray}
y^2+y+xy+q_1x+q_2x^2y=0\,.
\end{eqnarray}
The set of finite ramification points is
$R^{\circ}=\{0, {1\over 2}, {\tau\over 2}, {1+\tau\over 2}\}$.
The uniformization is given by the iteration of the following changes of coordinates (for some $\epsilon$ and $\kappa$)
\begin{equation}
\alpha=4^{1\over 3}\kappa^{2} \wp(u+\epsilon)-{1\over 3}({1\over 4}+{1\over 2}q_2)\,,\quad
\beta=\kappa^{3}\wp'(u+\epsilon)-\left(({1\over 2}-q_1)\alpha+{1\over 4}{q_2})\right)\,,
\end{equation}
\begin{equation}
x=\beta^{-1}{(\alpha+{q_2\over 2}-q_1^2+q_1)}\,,\quad
y=-{1\over 2}(1+x+q_2x^2)-{1\over 2}+x(x\alpha-({1\over 2}-q_1))\,.
\end{equation}
Taking the special one-parameter family $q_{1}=1,q_{2}=s$, we have
\begin{equation}
s=2^{-8} {\eta^{8}(\tau)\over \eta^{8}(4\tau)}\,,
\quad
\kappa=2^{-{13\over 3}}\pi^{-1} {\eta(2\tau)^2\over \eta(4\tau)^4}  \,.
\end{equation}

\subsection{$K_{\mathbb{P}^{1}\times \mathbb{P}^{1}}$}

Then affine part of the mirror curve given in Example \ref{ex:2} is equivalent to
\begin{eqnarray}
y^{2}+(1+x+q_{1}x^2) y+q_{2}x^{2}=0\,.
\end{eqnarray}
The set of finite ramification points is
$R^{\circ}=\{0, {1\over 2}, {\tau\over 2}, {1+\tau\over 2}\}$.
The uniformization is given by the iteration of the following changes of coordinates (for some $\epsilon$ and $\kappa$)
\begin{equation}
\alpha= 2^{2\over 3}\kappa^{2} \wp(u+\epsilon)+{1\over 12} (-1-2q_{1}+4q_{2})\,,\quad
\beta= \kappa^{3} \wp'(u+\epsilon)- {1\over 2} (\alpha+{1\over 2}q_{1})\,,
\end{equation}
\begin{equation}
x=   \beta^{-1}\left(2^{2\over 3}\kappa^{2} \wp(u+\epsilon) +{1\over 6} (1+2q_{1}-4q_{2})\right)\,,\quad
y=-{1\over 2}+ x (\alpha x -{1\over 2})  - {1\over 2} (1+x +q_{1} x^2)\,.
\end{equation}
Taking the special one-parameter
subfamily
$q_{1}=q_{2}=s$, we have
	\begin{equation}
	s=-2^{-8} {\eta^{8}(\tau)\over \eta^{8}(4\tau)}\,,
	\quad
	\kappa=2^{-{7\over 3}}\pi^{-1}\theta_{2}^{-2}(2\tau)\,.
	\end{equation}

\subsection{$K_{W\mathbb{P}[1,1,2]}$}

The affine part of the mirror curve given in Example \ref{ex:3} is equivalent to
\begin{equation}
y^2+{x^{4}}+y+ b_{4} x^{2}y+ b_{0} xy=0\,,
\quad
q_{1}=b_{4}b_{0}^{-4}\,,
q_{2}=b_{0}^{-2}\,.
\end{equation}
The set of finite ramification points is
$R^{\circ}=\{0, {1\over 2}, {\tau\over 2}, {1+\tau\over 2}\}$.
The following combination is independent of the specialization to an one-parameter subfamily
\begin{equation}
(b_{0}^2-4b_{4})^2
=
64 {   (\theta_{2}^{4}(2\tau) +\theta_{3}^{4} (2\tau))^{2}    \over  \theta_{4}^{8}(2\tau)}\,,
\end{equation}
up to an $SL_{2}(\mathbb{Z})$-transform on $\tau$.

The uniformization is given by the iteration of the following changes of coordinates (for some $\epsilon$ and $\kappa$)
\begin{equation}
\alpha=2^{3\over 2}\kappa^{2} \wp(u+\epsilon) -{1\over 12} (b_{0}^{2}+2 b_{4})\,,\quad
\beta=\kappa^{3}  \wp'(u+\epsilon)-{1\over 2}b_{0}  (\alpha+{1\over 2}b_{4})\,,
\end{equation}
\begin{equation}
x=\beta^{-1}\left(2^{3\over 2}\kappa^{2} \wp(u+\epsilon) +{1\over 3} (b_{4}-{1\over 4}b_{0}^{2}) \right)\,,\quad
y=-{1\over 2} +x (\alpha x-{1\over 2}b_{0})-{1\over 2} (1+b_{0}x+b_{4}x^2)\,.
\end{equation}
Taking the special one-parameter
subfamily
$(q_{1},q_{2})=(0,s)$ that is $b_{4}=0$, we have
	\begin{equation}
	s=64^{-1} {   \theta_{4}^{8}(2\tau) \over  (\theta_{2}^{4}(2\tau) +\theta_{3}^{4} (2\tau))^{2}    }\,,
	\quad
	\kappa=2^{-{1\over 3}} \pi^{-1} \theta_{4}^{-2} (2\tau)\,.
	\end{equation}

\end{appendices}


\providecommand{\bysame}{\leavevmode\hbox to3em{\hrulefill}\thinspace}
\providecommand{\MR}{\relax\ifhmode\unskip\space\fi MR }
\providecommand{\MRhref}[2]{%
  \href{http://www.ams.org/mathscinet-getitem?mr=#1}{#2}
}
\providecommand{\href}[2]{#2}

\bigskip{}

\noindent{\small Beijing International Center for Mathematical Research, Peking University,
5 Yiheyuan Road, Beijing 100871, China}

\noindent{\small Email: \tt bohanfang@gmail.com }

\medskip{}

\noindent{\small Department of Mathematics, University of Michigan, 2074 East Hall, 530 Church Street,
Ann Arbor, MI 48109, USA}

\noindent{\small  E-mail: \tt ruan@umich.edu}

\medskip{}

\noindent{\small Department of Mathematics, University of Michigan, 2074 East Hall, 530 Church Street,
	Ann Arbor, MI 48109, USA}

\noindent{\small E-mail: \tt yczhang15@pku.edu.cn}

\medskip{}

\noindent{\small Yau Mathematical Sciences Center, Jinchunyuan West Building,
Tsinghua University, Beijing 100084, China}

\noindent{\small Email: \tt jzhou2018@mail.tsinghua.edu.cn}

\end{document}